\documentclass[11pt,openany]{article}
\usepackage{amsmath}
\usepackage{amsthm}
\usepackage{amscd}
\usepackage{amssymb}
\usepackage{tikz,tikz-cd}
\usepackage{latexsym}
\newcommand{\cw}{{\curlywedge}}

\newcommand{\K}{{\mathbb K}}

\newcommand{\g}{\frak g}
\newcommand{\vv}{\frak v}

\newcommand{\h}{\frak h}
\newcommand{\ff}{\frak f}
\newcommand{\m}{\frak m}
\newcommand{\n}{\frak n}
\newcommand{\T}{\frak t}
\newcommand{\kk}{\frak k}
\newcommand{\f}{\frac}
\newcommand{\lo}{\longrightarrow}
\renewcommand{\r}{\frak r}

\renewcommand{\Im}{{\rm Im}}
\newtheorem{theorem}{Theorem}[section]

\newtheorem{corollary}[theorem]{Corollary}

\newtheorem{lemma}[theorem]{Lemma}

\newtheorem{proposition}[theorem]{Proposition}

\theoremstyle{definition}
\newcommand{\ds}{\displaystyle}
\newtheorem{definition}[theorem]{Definition}

\newtheorem{example}[theorem]{Example}

\title{\bf A non-abelian exterior product of Hom-Leibniz algebras}
\author{ Behrouz Edalatzadeh$^{a}$ , Seyedeh Narges Hosseini$^{b,}$\footnote{
Corresponding author.}
, Ali Reza Salemkar$^{b}$\\
{\small $^a$ Department of Mathematics, Faculty of Sciences, Razi
University, Kermanshah, Iran}\\
{\small $^b$ Department of Mathematics, Faculty of Mathematical
Sciences, Shahid Beheshti University, Tehran, Iran}\\
{\small edalatzadeh@gmail.com~,~narges.hosseini90@gmail.com~,~
salemkar@sbu.ac.ir}}
\date{ }
\begin{document}
\maketitle
\begin{abstract}
In this paper, we introduce a non-abelian exterior product of
Hom-Leibniz algebras and investigate its relative to the Hopf's
formula. We also construct an eight-term exact sequence in the
homology of Hom-Leibniz algebras. Finally, we relate the notion of
capability of a Hom-Leibniz algebra to its exterior product.
\\[.2cm]
\end{abstract}

\section{Introduction}
Leibniz algebras have been initially introduced by Bloh \cite{B}
as a non skew-symmetric analogue of Lie algebras and further
explored by Loday \cite{L} for constructing a new (co)homology
theory for Lie algebras, the so called Leibniz (co)homology (see
also \cite{LP}). In \cite{G}, Gnedbaye introduced the notion of
the non-abelian tensor product of Leibniz algebras and gave some
of its fundamental properties. Using this notion, Donadze et al.
\cite{DMK} defined the non-abelian exterior product of these
algebras and investigated its relations to the low dimensional
Leibniz homology. In particular, they obtained the Hopf formulas
for the (higher) homology of Leibniz algebras and, using the
theory of simplicial Leibniz algebras, described the second
homology of Leibniz algebras as central subalgebras of their
exterior products. Moreover, they constructed an eight term exact
sequence in Leibniz homology. One of the important tools in
studying the classification of Leibniz algebras is the use of
capability property, which is closely related to the second
homology of Leibniz algebras. In \cite{HES}, the authors dealt
with some of the applications of exterior product to the notion of
capability of Leibniz algebras. There is a series of papers (for
instance, see \cite{CK,EH,EP,EV,HES1,KKL}) emphasizing the
relevance of the tensor and exterior products to the development
and exposition of the basic theories of the capability and the
homology of Leibniz algebras.\vspace{.2cm}

Hom-Lie algebras have been introduced by Hartwig, Larson and
Silvestrov \cite{HLS} to study the deformations of the Witt and
Virasoro algebras. This notion mainly arises from the
investigation of quantum deformations and discretisation of vector
fields via twisted derivations (\cite{LS}); for more information
on Hom-Lie algebras and further references concerning the subject
we refer the reader to \cite{CG,CIP,CKR,JL,S,SC,SX}. After
introducing Hom-Lie algebras, many authors extended this idea to
several other algebraic structures. In particular, Makhlouf and
Silvestrov \cite{MS,MS1} introduced the notion of Hom-Leibniz
algebras, which is a natural generalization of Hom-Lie and Leibniz
algebras. A Hom-Leibniz algebra is a vector space, say $\g$,
endowed with a bilinear bracket satisfying a Leibniz identity
twisted by a linear self-map $\alpha_\g$ of $\g$. If the bracket
is skew-symmetric or $\alpha_\g$ is the identity map, then we
recover the definitions of Hom-Lie algebra and Leibniz algebra,
respectively. Hence, one is motivated to investigate some
appropriate results of the theories of Hom-Lie and Leibniz
algebras in the general setting of the category of Hom-Leibniz
algebras. Accordingly, Cheng and Su \cite{CS} developed some
structure theory, such as (co)homology groups and universal
central extensions, for Hom-Leibniz algebras (see also
\cite{CIR}). Casas and Rego \cite{CR} introduced the notions of
Hom-action and semi-direct product of  Hom-Leibniz algebras and
established the equivalence between split extensions and the
semi-direct product extensions of these algebras. Recently, Cases,
Khmaladze and Rego \cite{CKR1} introduced the notion of the
non-abelian tensor product of Hom-Leibniz algebras and described
universal central extensions of  these algebras via Hom-Leibniz
tensor product. In this paper, we introduce the notion of the
non-abelian exterior product of Hom-Leibniz algebras. We obtain
some structure properties of the exterior products of Hom-Leibniz
algebras, and use them to give several homology results,
generalizing known ones for Leibniz algebras. In particular, the
Hopf's formula for Hom-Leibniz algebras is proved. We also
introduce the notion of capability of Hom-Leibniz algebras and
investigate the relationship between the exterior product and the
capability.\vspace{.15cm}

This paper is organized as follows: In Section 2, we review some
basic definitions on Hom-Leibniz algebras, with a special mention
of the Hom-Leibniz action and the Leibniz homology, and give some
necessary results for the development of the paper. In Section 3,
we define and gather general results on the exterior product of
Hom-Leibniz algebras. In Section 4, several key results from the
homology theory of Leibniz algebras are extended to Hom-Leibniz
algebras. In particular, we describe the second homologies of
Hom-Leibniz algebras as central subalgebras of their exterior
products, and obtain an exact sequence of eight terms associated
with an extension of Hom-Leibniz algebras in the homology. In the
last section, we show that to investigate the capability of
non-perfect Hom-Lie algebras, we only need to study those
Hom-Leibniz algebras whose companion endomorphisms are surjective.
We also establish the capability property for Hom-Leibniz algebras
which are perfect or abelian. Furthermore, under some conditions,
we show that capability is inherited
to both the two summands.\vspace{.2cm}\\
{\large{\bf Notations}}. Throughout this paper, all vector spaces
and algebras are considered over some fixed field $\K$ and linear
maps are $\K$-linear maps.  We write $\otimes$ and $\wedge$ for
the usual tensor and exterior products of vector spaces over $\K$,
respectively. For any vector space (respectively, Hom-Leibniz
algebra) $\g$, a subspace (respectively, an ideal) $\g'$ and
$x\in\g$, we write $\bar x$ to denote the coset $x+\g'$. If
${\frak a}$ and ${\frak b}$ are subspaces of a vector space $\vv$
for which $\vv={\frak a}+{\frak b}$ and ${\frak a}\cap{\frak
b}=0$, we will write $\vv={\frak a}\dot{+}{\frak b}$.
\section{Preliminaries on Hom-Leibniz algebras}
This section is devoted to recall some essential definitions and
auxiliary results, which will be needed.
\subsection{Basic definitions}
A {\it Hom-Leibniz algebra} $(\g,\alpha_\g)$ is a non-associative
algebra $\g$ together with a linear map
$\alpha_\g:\g\longrightarrow \g$ satisfying the condition

~\hspace{3cm}$[\alpha_\g(x),[y,z]]=[[x,y],\alpha_\g(z)]-[[x,z],\alpha_\g(y)]$,\hspace{2cm}
(Hom-Leibniz~identity)\\
for all $x,y,z\in\g$, where $[~,~]$ denotes the product in $\g$.
In the whole paper we only deal with (the so-called {\it
multiplicative}) Hom-Leibniz algebras $(\g,\alpha_\g)$ such that
$\alpha_\g$ preserves the product, that is,
$\alpha_\g([x,y])=[\alpha_\g(x),\alpha_\g(y)]$ for all $x,y\in\g$.
The Hom-Leibniz algebra $(\g,\alpha_\g)$ is said to be {\it
regular} if $\alpha_\g$ is bijective. Taking $\alpha_\g=id_\g$, we
recover exactly Leibniz algebras. A {\it Hom-vector space} is a
pair $(\vv,\alpha_\vv)$, where $\vv$ is a vector space and
$\alpha_\vv$ is a linear self-map of $\vv$. If we put $[x,y]=0$
for all $x,y\in\vv$, then $(\vv,\alpha_\vv)$ is a Hom-Leibniz
algebra, which is called an {\it abelian Hom-Leibniz algebra}.

A {\it homomorphism} of Hom-Leibniz algebras
$\delta:(\g_1,\alpha_{\g_1})\longrightarrow(\g_2,\alpha_{\g_2})$
is an algebra homomorphism from $\g_1$ to $\g_2$ such that $\delta
\circ\alpha_{\g_1}=\alpha_{\g_2}\circ\delta$. The corresponding
category of Hom-Leibniz algebras is denoted by $\bf{HomLb}$. As
this category is a Jonsson-Tarski algebraic variety with a unique
constant, it is a semi-abelian category, and then the well-known
Snake Lemma is valid for Hom-Leibniz algebras. We refer the reader
to \cite{CKR1} for obtaining more information on this category.

Let $(\g,\alpha_\g)$ be an arbitrary Hom-Leibniz algebra. Then\\
$\bullet$\hspace{.3cm} A (Hom-Leibniz) {\it subalgebra}
$(\h,\alpha_\h)$ of $( \g,\alpha_\g)$ consists of a vector
subspace $\h$ of $\g$, which is closed under the product and
invariant under the map $\alpha_\g$, together with the linear
self-map $\alpha_\h$ being the restriction of $\alpha_\g$ on $\h$.
In such a case we may write $\alpha_\g|$ for $\alpha_\h$. A
subalgebra $(\h,\alpha_\h)$
is an {\it ideal} if $[x,y],[y,x]\in\h$ for all $x\in\h$, $y\in\g$.\\
$\bullet$\hspace{.3cm} The {\it center} of $(\g,\alpha_\g)$ is the
vector space $Z(\g)=\{x\in\g~|~[x,y]=0$, for all $y\in\g\}$.\\
Put $Z_{\alpha}(\g)=\{x\in\g~|~[\alpha_\g^n(x),y]=0,$ for all
$y\in\g, n\geq0\}$, where $\alpha_\g^0=id_\g$ and $\alpha_\g^n$,
$n\geq1$, denotes the composition of $\alpha_\g$ with itself $n$
times. It is straightforward to see that
$(Z_{\alpha}(\g),\alpha_{Z_{\alpha}(\g)})$ is the largest central
ideal of $(\g,\alpha_\g)$. We call $Z_{\alpha}(\g)$ the
$\alpha$-{\it center} of $(\g,\alpha_\g)$. When $\alpha_L$ is a
surjective homomorphism or $(\g,\alpha_\g)$ is abelian, then
$Z_{\alpha}(\g)=Z(\g)$.\\
$\bullet$\hspace{.3cm} If $(\h,\alpha_\h)$ and $(\kk,\alpha_\kk)$
are two ideals of $(\g,\alpha_\g)$, then the {\it$($Higgins$)$
commutator} of $(\h,\alpha_\h)$ and $(\kk,\alpha_\kk)$, denoted by
$([\h,\kk],\alpha_{[\h,\kk]})$, is a Hom-Leibniz subalgebra
generated by the elements $[h,k]$, $h\in\h$, $k\in\kk$. Note that
$([\h,\kk],\alpha_{[\h,\kk]})$ is an ideal of both
$(\h,\alpha_\h)$ and $(\kk,\alpha_\kk)$. Especially,
$([\g,\g],\alpha_{[\g,\g]})$ is an ideal of $(\g,\alpha_\g)$. The
quotient $(\g/[\g,\g],\bar\alpha_\g)$ is called the {\it
abelianisation} of $(\g,\alpha_\g)$, and denoted by
$(\g^{ab},\alpha_{\g^{ab}})$. The Hom-Leibniz algebra
$(\g,\alpha_\g)$
is called {\it perfect} if $\g=[\g,\g]$.\\
$\bullet$\hspace{.3cm} An exact sequence
$(\m,\alpha_\m)\rightarrowtail(\kk,\alpha_\kk)\twoheadrightarrow(\g,\alpha_\g)$
of Hom-Leibniz algebras is a {\it central extension} of
$(\g,\alpha_\g)$ if $\m\subseteq Z(\kk)$, or equivalently,
$[\m,\kk]=0$.\vspace{.2cm}

Let $(\m,\alpha_\m)$ and $(\n,\alpha_\n)$ be two Hom-Leibniz
algebras. By a {\it Hom-Leibniz action} of $(\m,\alpha_\m)$ on
$(\n,\alpha_\n)$, we mean a couple of $\K$-bilinear maps $\m\times
\n\to\n, (m,n)\mapsto\hspace{-.17cm}~^mn$ and $\n\times\m\to\n$,
$(n,m)\mapsto n^m$, satisfying the axioms:
\begin{alignat*}{2}
    &(A1)~^{[m,m']}\alpha_\n(n)=(~^mn)^{\alpha_\m(m')}+~^{\alpha_\m(m)}(n^{m'}), ~~~~~~~~
    &&\hspace{-.5cm}(A2)~~^{\alpha_\m(m)}[n,n']=[~^mn,\alpha_\n(n')]-[~^mn',\alpha_\n(n)],\\
    &(A3)~\alpha_\n(n)^{[m,m']}=(n^m)^{\alpha_\m(m')}-(n^{m'})^{\alpha_\m(m)},
    &&\hspace{-.5cm}(A4)~[n,n']^{\alpha_\m(m)}=[n^m,\alpha_\n(n')]+[\alpha_\n(n),{n'}^m],\\
    &(A5)~~^{\alpha_\m(m)}(~^{m'}n)=-^{\alpha_\m(m)}(n^{m'}),
    &&\hspace{-.5cm}(A6)~[\alpha_\n(n),~^mn']=-[\alpha_\n(n),{n'}^m],\\
    &(A7)~\alpha_\n(~^mn)=~^{\alpha_\m(m)}\alpha_\n(n),
    &&\hspace{-.5cm}(A8)~\alpha_\n(n^m)=\alpha_\n(n)^{\alpha_\m(m)},
\end{alignat*}
for all $m,m'\in\m, n,n'\in\n$. The action is called {\it trivial}
if $~^mn=n^m=0$, for all $m\in\m$, $n\in\n$. Also, if
$(\n,\alpha_\n)$ is an abelian Hom-Leibniz algebra enriched with a
Hom-Leibniz action of $(\m,\alpha_\m)$, then $(\n,\alpha_\n)$ is
said to be a {\it Hom-co-representation} of  $(\m,\alpha_\m)$
(\cite{CKR1}). It is obvious that if $(\m,\alpha_\m)$ is a
subalgebra of some Hom-Leibniz algebra $(\g,\alpha_\g)$ and
$(\n,\alpha_\n)$ is an ideal of $(\g,\alpha_\g)$, then the product
in $\g$ induces a Hom-Leibniz action of $(\m,\alpha_\m)$ on
$(\n,\alpha_\n)$ given by$~^mn=[m,n]=m^n$. In particular, there is
a Hom-Leibniz action of $(\g,\alpha_\g)$ on itself given by the
product in $\g$.

Let $(\m,\alpha_\m)$ and $(\g,\alpha_\g)$ be Hom-Leibniz algebras
together with a Hom-Leibniz action of $(\g,\alpha_\g)$ on
$(\m,\alpha_\m)$. Then the {\it semi-direct product} $(\m\rtimes
\g,\alpha_{\rtimes})$ is defined to be the Hom-Leibniz algebra
with the underlying vector space $\m\dot{+}\g$, the endomorphism
$\alpha_{\rtimes}:\m\rtimes\g\lo\m\rtimes\g$ given by
$\alpha_{\rtimes}(m,x)=(\alpha_\m(m),\alpha_\g(x))$, and the
product\vspace{-.15cm}
\[[(m_1,x_1),(m_2,x_2)]=([m_1,m_2]+\vspace{-.1cm}~^{\alpha_\g(x_1)}m_2
+m_1^{\alpha_\g(x_2)}, [x_1,x_2]).\vspace{-.1cm}\] When
$(\g,\alpha_\g)$ acts trivially on $(\m,\alpha_\m)$, we get the
direct sum structure of Hom-Leibniz algebras.

Let $(\m,\alpha_\m)\stackrel{i}\rightarrowtail(\kk,\alpha_\kk)
\stackrel{\zeta}\twoheadrightarrow(\g,\alpha_\g)$ be a split short
exact sequence of Hom-Leibniz algebras, that is, there is a
homomorphism $\eta:(\g,\alpha_\g)\lo(\kk,\alpha_\kk)$ such that
$\zeta\circ\eta=id_\g$. Then there is a Hom-Leibniz action of
$(\g,\alpha_\g)$ on $(\m,\alpha_\m)$ defined by
$~^xm=i^{-1}[\eta(x),i(m)]$ and $m^x=i^{-1}[i(m),\eta(x)]$ for all
$m\in\m$, $x\in\g$. Moreover, we have the following commutative
diagram with exact rows.\vspace{.2cm}

\tikzset{node distance=3cm, auto}$~~~~$
\begin{tikzpicture}[%
back line/.style={densely dotted}, cross
line/.style={preaction={draw=white, -,line width=6pt}}]
\hspace{1.2cm}
    \node(A1){\fontsize{9.5}{5}\selectfont$~$};
    \node[right of=A1](B1){\fontsize{9.5}{5}\selectfont$(\m,\alpha_\m)$};
    \node[right of=B1](C1){\fontsize{9.5}{5}\selectfont$(\m\rtimes\g,\alpha_{\rtimes})$};
    \node[right of=C1](D1){\fontsize{9.5}{5}\selectfont$(\g,\alpha_\g)$};
    \node(B2)[below of=B1, node distance=1.7cm]{\fontsize{9.5}{5}\selectfont$(\m,\alpha_{\m})$};
    \node(C2)[below of=C1, node distance=1.7cm]{\fontsize{9.5}{5}\selectfont$(\kk,\alpha_\kk)$};
    \node(D2)[below of=D1, node distance=1.7cm]{\fontsize{9.5}{5}\selectfont$(\g,\alpha_{\g}),$};

    \draw[>->](B1) to node{$i$}(C1);
    \draw[>->>](B1) to node[right]{\fontsize{9.5}{5}\selectfont$id_\m$}(B2);
    \draw[>->](B2) to node{\fontsize{9.5}{5}\selectfont$\subseteq$}(C2);
    \draw[>->>](C1) to node[right]{\fontsize{9.5}{5}\selectfont$\xi$}(C2);
    \draw[->>](C1) to node{\fontsize{9.5}{5}\selectfont$\rho$}(D1);
    \draw[->>](C2) to node{\fontsize{9.5}{5}\selectfont$\xi$}(D2);
    \draw[>->>](D1) to node[right]{\fontsize{9.5}{5}\selectfont$id_\g$}(D2);
\end{tikzpicture}\\
where $i$ and $\rho$ are the canonical and projective Hom-maps,
respectively, and $\xi$ is an isomorphism defined by
$\xi(m,x)=m+\eta(x)$. In particular, if we put $\T=\eta(\g)$, then
$(\T,\alpha_\T)$ is a subalgebra of $(\kk,\alpha_\kk)$ such that
$\kk=\m\dot{+}\T$.

A {\it crossed module} of Hom-Leibniz algebras is a homomorphism
$\mu:(\m,\alpha_\m)\to(\g,\alpha_\g)$ together with a Hom-Leibniz
action of $(\g,\alpha_\g)$ on $(\m,\alpha_\m)$ such
that\vspace{-.15cm}
\[\mu(~^xm)=[x,\mu(m)]~~,~~\mu(m^x)=[\mu(m),x]~~,
~~^{\mu(m)}{m'}=[m,m']= m^{\mu(m')}\vspace{-.15cm}\] for all
$m,m'\in\m, x\in\g$. It is obvious that if $(\m,\alpha_\m)$ is an
ideal of $(\g,\alpha_\g)$, then the inclusion Hom-map
$(\m,\alpha_\m)\hookrightarrow(\g,\alpha_\g)$ is a crossed module.
\subsection{Free Hom-Leibniz algebras}
Let $(\vv,\alpha_\vv)$ be a Hom-vector space,
$T(\vv)=\bigoplus_{n\geq1}\vv^{\otimes n}$ the tensor space over
$\vv$ and $\bar\alpha_\vv=\bigoplus_{n\geq1}\alpha_\vv^{\otimes
n}$. Then it is easy to see that $(T(\vv),\bar\alpha_\vv)$ is a
Hom-vector space. Also, $(T(\vv),\bar\alpha_\vv)$ equipped with
the product defined by $x\cdot y=x\otimes y$ is a multiplicative
Hom-algebra. In fact, $(T(\vv),\bar\alpha_\vv)$ is the free
multiplicative algebra over $(\vv,\alpha_\vv)$. Let ${\frak
i}_\vv$ be an ideal of $(T(\vv),\bar\alpha_\vv)$ generated by the
elements\vspace{-.1cm}
$$\bar\alpha_\vv(x)\cdot(y\cdot z)-(x\cdot y)\cdot
\bar\alpha_\vv(z)+(x\cdot
z)\cdot\bar\alpha_\vv(y),\vspace{-.1cm}$$ for all $x,y,z\in
T(\vv)$. It is clear that the quotient algebra $(T(\vv)/{\frak
i}_\vv,\tilde\alpha_\vv)$ is a Hom-Leibniz algebra, where
$\tilde\alpha_\vv$ is induced by $\bar\alpha_\vv$. We denote the
factor algebra $(T(\vv)/{\frak i}_\vv,\tilde\alpha_\vv)$ by
$(\ff_\vv,\alpha_\ff)$. The proof of the following universal
property is straightforward.
\begin{proposition}
Let $(\g,\alpha_\g)$ be a Hom-Leibniz algebra, $(\vv,\alpha_\vv)$
be a Hom-vector space, and $f\colon(\vv,\alpha_\vv) \to
(\g,\alpha_\g)$ a homomorphism of Hom-vector spaces. Then there
exits a unique homomorphism of Hom-Leibniz algebras $\bar
f:(\ff_\vv,\alpha_\ff)\to(\g,\alpha_\g)$, that is,
$(\ff_\vv,\alpha_\ff)$ is the free Leibniz algebra over
$(\vv,\alpha_\vv)$.
\end{proposition}
There is an adjoint pair of functors\vspace{-.2cm}
$$
\begin{tikzcd}
\bold{HomVect}\arrow[r, shift left=.75ex, "G"{name=G}] &
\bold{HomLb}\arrow[l, shift left=.75ex, "U"{name=F}]
\arrow[phantom, from=F, to=G, "\dashv" rotate=-90].
\end{tikzcd}\vspace{-.2cm}
$$
where $G$ sends a Hom-vector space $(\vv,\alpha_\vv)$ to the free
Hom-Leibniz algebra $(\ff_\vv,\alpha_\ff)$ and $U$ sends a
Hom-algebra $({\frak a},\alpha_{\frak a})$ to the Hom-vector space
obtained by forgetting the operations. As an immediate
consequence, we obtain that every Hom-Leibniz algebra admits at
least one free presentation.

Let $(\g,\alpha_\g)$ be a Hom-Leibniz algebra with a free
presentation
$(\r,\alpha_\r)\rightarrowtail(\ff,\alpha_\ff)\twoheadrightarrow(\g,\alpha_\g)$.
Then the {\it Schur multiplier} of $(\g,\alpha_\g)$ is defined to
be the abelian Hom-Leibniz algebra\vspace{-.05cm}
\[{\cal{M}}(\g,\alpha_\g)=\frac{(\r,\alpha_\r)\cap
([\ff,\ff],\alpha_{[\ff,\ff]})}{([\ff,\r],\alpha_{[\ff,\r]})}.\vspace{-.05cm}\]
As the category of Hom-Lie algebras is semi-abelian, Theorem 6.9
of \cite{EV1} indicates that the Schur multiplier of
$(\g,\alpha_\g)$ is independent of the choice of the free
presentation of $(\g,\alpha_\g)$.\vspace{.2cm}

The homology of Hom-Leibniz algebras, which is a generalization of
the Chevalley-Eilenberg homology of Leibniz algebras, is
constructed as follows: Let $(\g,\alpha_\g)$ be a Hom-Leibniz
algebra and $(\m,\alpha_\m)$ be a Hom-co-representation of
$(\g,\alpha_\g)$. The homology of $(\g,\alpha_\g)$ with
coefficient in $(\m,\alpha_\m)$, denoted by $HL_*^\alpha(\g,\m)$,
is the homology of the Hom-chain complex
$(CL_*^\alpha(\g,\m),d_*)$, where $CL_n^\alpha(\g,\m):=(\m\otimes
\g^{\otimes n},\alpha_\m\otimes \alpha_\g^{\otimes n})$, $n\geq0$
($\g^{\otimes n}$ denotes the $n$-th tensor power of $\g$ with
$\g^{\otimes0}=\K$) and the boundary map
$d_n:CL_n^\alpha(\g,\m)\to CL_{n-1}^\alpha(\g,\m)$, is a
homomorphism of Hom-vector spaces given by
\begin{align*}
\hspace{1cm}d_n(&m\otimes x_1\otimes\dots\otimes x_n)=m^{x_1}
\otimes\alpha_\g(x_2)\otimes\dots\otimes\alpha_\g(x_n)\\
&+\sum_{i=2}^n (-1)^i
~^{x_i}m\otimes\alpha_\g(x_1)\otimes\dots\otimes
\widehat{\alpha_\g(x_i)}\otimes\dots\otimes\alpha_\g(x_n)\\
&+\sum_{1\leq i<j\leq n} (-1)^{j+1}\alpha_\m(m)\otimes
\alpha_\g(x_1)\otimes\dots
\otimes\alpha_\g(x_{i-1})\otimes[x_i,x_j]\otimes\alpha_\g(x_{i+1})\\
&\hspace{.75cm}\otimes\dots\otimes\widehat{\alpha_\g(x_j)}\otimes\dots\otimes\alpha_\g(x_n),
\end{align*}
where the notation $\widehat{\alpha_\g(x_i)}$ means that the
variable $\alpha_\g(x_i)$ is omitted. It is easy to see that
$HL_n^\alpha(\g,\m)=(\ker(d_n)/\Im(d_{n-1}),\overline{\alpha_\m\otimes
\alpha_\g^{\otimes n}})$ is a Hom-vector space for all $n\geq1$
(note that $HL_{n}^{\alpha}(\g,\m)$ is defined in \cite{Y} to be
the vector space $\ker(d_n)/\Im(d_{n-1})$, while we here take it
as a Hom-vector space). In a special case, if
$(\m,\alpha_\m)=(\K,id_\K)$ is a trivial Hom-co-representation of
$(\g,\alpha_\g)$, then $HL_n^\alpha(\g,\K)$ is called the {\it
n-th homology} of $(\g,\alpha_\g)$ and denoted by
$HL_n^\alpha(\g)$. It is easily verified that there is an
isomorphism of Hom-vector spaces $HL_1^\alpha(\g)\cong
(\g^{ab},\bar\alpha_{\g^{ab}})$. In Section 4, we give a relation
between the second homology of $(\g,\alpha_\g)$ and ${\cal
M}(\g,\alpha_\g)$.
\section{The exterior product of Hom-Leibniz algebras}
In this section, we introduce the notion of the (non-abelian)
exterior product of Hom-Leibniz algebras, which generalizes the
exterior product of Leibniz algebras \cite{DMK}, and study some of
its basic properties. Let us first recall from \cite{CKR1} the
definition of the (non-abelian) tensor product of Hom-Leibniz
algebras.

Let $\partial_1:(\m,\alpha_\m)\to(\g,\alpha_\g)$ and
$\partial_2:(\n,\alpha_\n)\to(\g,\alpha_\g)$ be two crossed
modules of Hom-Leibniz algebras. Then $(\m,\alpha_\m)$ and
$(\n,\alpha_\n)$ act compatibility on each other via the
Hom-Leibniz action of $(\g,\alpha_\g)$.  The {\it
$($Hom-Leibniz$)$ tensor product} $(\m\ast\n,\alpha_{\m\ast\n})$
is defined as the Hom-Leibniz algebra generated by the symbols
$m\ast n$ and $n\ast m$ (for $m\in\m$, $n\in\n$) subject to the
following defining relations\vspace{-.2cm}
\begin{alignat*}{2}
    &(1a)~k(m\ast n)=k m\ast n=m\ast kn,~~~~~~
    &&(4a)~\alpha_\m(m)\ast[n,n']=m^n\ast\alpha_\n(n')-m^{n'}\ast\alpha_\n(n),\\
    &(1b)~k(n\ast m)=kn\ast m=n\ast km,~~~~~~
    &&(4b)~\alpha_\n(n)\ast[m,m']=n^m\ast\alpha_\m(m')-n^{m'}\ast\alpha_\m(m),\\
    &(2a)~(m+m')\ast n=m\ast n+m'\ast n,~~~~~~
    &&(4c)~[m,m']\ast\alpha_\n(n)=\hspace{-.1cm}~^mn\ast\alpha_\m(m')-\alpha_\m(m)\ast
    n^{m'},\\
    &(2b)~(n+n')\ast m=n\ast m+n'\ast m,~~~~~~
    &&(4d)~[n,n']\ast\alpha_\m(m)=\hspace{-.1cm}~^nm\ast\alpha_\n(n')-\alpha_\n(n)\ast
    m^{n'},\\
    &(2c)~m\ast(n +n')=m\ast n+m\ast n',~~~~~~
    &&(5a)~m^n\ast\hspace{-.1cm}~^{m'}n'=[m\ast n,m'\ast n']=\hspace{-.1cm}~^mn\ast
    m'^{n'},\\
    &(2d)~n\ast(m+m')=n\ast m+n\ast m',~~~~~~
    &&(5b)~~^nm\ast n'^{m'}=[n\ast m,n'\ast m']=n^m\ast\hspace{-.1cm}~^{n'}m',\\
    &(3a)~\alpha_\m(m)\ast\hspace{-.1cm}~^{m'}n=-\alpha_\m(m)\ast
    n^{m'},~~~~~~
    &&(5c)~m^n\ast n'^{m'}=[m\ast n,n'\ast
    m']=\hspace{-.1cm}~^mn\ast\hspace{-.1cm}~^{n'}m',\\
    &(3b)~\alpha_\n(n)\ast\hspace{-.1cm}~^{n'}m=-\alpha_\n(n)\ast
    m^{n'},~~~~~~
    &&(5d)~~^nm\ast\hspace{-.1cm}~^{m'}n'=[n\ast m,m'\ast n']=n^m\ast
    m'^{n'},
\end{alignat*}
for all $k\in\K$, $m,m'\in\m$, $n,n'\in\n$, and the endomorphism
$\alpha_{\m\ast\n}$ is given on generators by\vspace{-.2cm}
\[\alpha_{\m\ast\n}(m\ast
n)=\alpha_\m(m)\ast\alpha_\n(n)~~~~{\rm
and}~~~~\alpha_{\m\ast\n}(n\ast
m)=\alpha_\n(n)\ast\alpha_\m(m).\vspace{-.2cm}\] Note that the
identity map $id_\g:(\g,\alpha_\g)\lo (\g,\alpha_\g)$ is a crossed
module with $(\g,\alpha_\g)$ acting on itself by the product, so
we can form the tensor products $(\g\ast \m,\alpha_{\g\ast\m})$,
$(\g\ast\n,\alpha_{\g\ast\n})$ and $(\g\ast\g,\alpha_{\g\ast\g})$.
Also, if $\alpha_\m=id_\m$, $\alpha_\n=id_\n$ and
$\alpha_\g=id_\g$, then $\m\ast\n$ coincides with the tensor
product of Leibniz algebras given in \cite{G}.

The following proposition summarizes the rather elementary
properties of Hom-Leibniz tensor products, the proof of which are
left to the reader (see also \cite{CKR1}).
\begin{proposition}\label{1}
With the above assumptions and notations, we have $:$

$(i)$ The maps\vspace{-.4cm}
\begin{alignat*}{1}
\lambda&:(\m\ast\n,\alpha_{\m\ast\n})\lo(\g,\alpha_\g),~~~~~~~\lambda(m\ast
n)=[\partial_1(m),\partial_2(n)],~~~~\lambda(n\ast
m)=[\partial_2(n),\partial_1(m)],\\[-.17cm]
\lambda_\m&:(\m\ast\n,\alpha_{\m\ast\n})\lo(\m,\alpha_\m),~~~~\lambda_\m(m\ast
n)=m^n,~~~~\lambda_\m(n\ast
m)=\hspace{-.1cm}~^nm,\\[-.17cm]
\lambda_\n&:(\m\ast\n,\alpha_{\m\ast\n})\lo(\n,\alpha_\n),\hspace{.1cm}~~~~~\lambda_\n(m\ast
n)=\hspace{-.1cm}~^mn,~~~~\lambda_\n(n\ast m)=n^m,
\end{alignat*}
are homomorphisms of Hom-Leibniz algebras with the kernels
contained in the center of $(\m\ast\n,\alpha_{\m\ast\n})$.

$(ii)$ There is a Hom-Leibniz action of $(\g,\alpha_\g)$ on
$(\m\ast\n,\alpha_{\m\ast\n})$ given by\vspace{-.1cm}
\begin{alignat*}{2}
    &~^{x}(m\ast
    n)=\hspace{-.1cm}~^{x}m\ast\alpha_\n(n)-\hspace{-.1cm}~^{x}n\ast\alpha_\m(m),~~~~~~
    &&~(m\ast n)^{x}=m^x\ast\alpha_\n(n)+\alpha_\m(m)\ast
    n^{x},\\
    &~^{x}(n\ast m)=\hspace{-.1cm}~^{x}n\ast\alpha_\m(m)-\hspace{-.1cm}~^{x}m\ast\alpha_\n(n),~~~~~~
    &&~(n\ast m)^{x}=n^{x}\ast\alpha_\m(m)+\alpha_\n(n)\ast m^x,
\end{alignat*}
and then $(\m,\alpha_\m)$ and $(\n,\alpha_\n)$ act on
$(\m\ast\n,\alpha_{\m\ast\n})$ via $\partial_1$ and $\partial_2$.
Moreover,\vspace{-.2cm}
\begin{equation}
\lambda(\hspace{-.1cm}~^xy)=[\alpha_\g(x),\lambda(y)]~~,~~
\lambda(y^x)=[\lambda(y),\alpha_\g(x)]~~,~~^{\lambda(y')}y=[\alpha_{\m\ast\n}(y'),y]~~,~~
y^{\lambda(y')}=[y,\alpha_{\m\ast\n}(y')]~~~~~\vspace{-.2cm}
\end{equation}
for all $y,y'\in\m\ast\n$, $x\in\g$, and the relations similar to
$(1)$ are valid for $\lambda_\m$ and $\lambda_\n$.

$(iii)$ If $(\m,\alpha_\n)$ and $(\n,\alpha_\n)$ act trivially on
each other, and both endomorphisms $\alpha_\m$ and $\alpha_\n$ are
surjective, then there is an isomorphism of abelian Leibniz
algebras\vspace{-.1cm}
\[(\m\ast\n,\alpha_{\m\ast\n})\cong((\m^{ab}\otimes\n^{ab})\oplus
(\n^{ab}\otimes\m^{ab}),\alpha_{\oplus}),\vspace{-.1cm}\] where
$\alpha_{\oplus}$ denotes the linear self-map of
$(\m^{ab}\otimes\n^{ab})\oplus(\n^{ab}\otimes\m^{ab})$ induced by
$\alpha_\m$ and $\alpha_\n$.

$(iv)$ Let
$(\m,\alpha_\m)\rightarrowtail(\kk,\alpha_\kk)\twoheadrightarrow(\g,\alpha_\g)$
and
$(\m',\alpha_{\m'})\rightarrowtail(\kk',\alpha_{\kk'})\twoheadrightarrow(\g',\alpha_{\g'})$
are short exact sequences of Hom-Leibniz algebras, where
$(\m',\alpha_{\m'})$  and $(\kk',\alpha_{\kk'})$ are ideals of
$(\kk,\alpha_\kk)$, and $(\g',\alpha_{\g'})$ is an ideal of
$(\g,\alpha_\g)$. Then there exists an exact sequence of
Hom-Leibniz algebras\vspace{-.2cm}
\[((\m\ast\kk')\rtimes(\kk\ast\m'),\alpha_{\rtimes})\lo(\kk\ast\kk',\alpha_{\kk\ast\kk'})
\twoheadrightarrow(\g\ast\g',\alpha_{\g\ast\g'}),\vspace{-.2cm}\]
in which the Hom-Leibniz action of
$(\kk\ast\m',\alpha_{\kk\ast\m'})$ on
$(\m\ast\kk',\alpha_{\m\ast\kk'})$ is induced by the homomorphism
$\lambda_{\m'}:(\kk\ast\m',\alpha_{\kk\ast\m'})\lo(\m',\alpha_{\m'})$.
\end{proposition}
According to part $(i)$ of the above proposition, for any
Hom-Leibniz algebra $(\g,\alpha_\g)$, the Hom-map
$\lambda_\g:(\g\ast\g,\alpha_{\g\ast\g})\lo(\g,\alpha_\g)$,
$x_1\ast x_2\longmapsto[x_1,x_2]$, is a homomorphism (which is
called the {\it commutator Hom-map}). We set
$JL_2^{\alpha}(\g)=(\ker(\lambda_\g),\alpha_{\ker(\lambda_\g)})$.\vspace{.2cm}

Let $\m\square\n$ be the vector subspace of $\m\ast\n$ spanned by
the elements of the form $m\ast n'-n\ast m'$ with
$\partial_1(m)=\partial_2(n)$ and $\partial_1(m')=\partial_2(n')$.
Note that $\m\square\n$ lies in the $\alpha$-center of
$(\m\ast\n,\alpha_{\m\ast\n})$, because for any $m\ast n'-n\ast
m'\in\m\square\n$, $m_1\in\m$, $n_1\in\n$, we have\vspace{-.2cm}
\begin{align*}
\hspace{-.5cm}[\alpha^k_{\m\ast\n}(m\ast n'-n\ast m'),m_1\ast
n_1]=&\alpha^k_\m(m)^{\alpha^k_\n(n')}\ast~^{m_1}n_1-
~^{\alpha^k_\n(n)}\alpha^k_\m(m')\ast~^{m_1}n_1\\
=&\alpha^k_\m(m)^{\partial_2(\alpha^k_\n(n'))}\ast~^{m_1}n_1-
~^{\partial_2(\alpha^k_\n(n))}\alpha^k_\m(m')\ast~^{m_1}n_1\\
=&\alpha^k_\m(m)^{\alpha^k_\g(\partial_2(n'))}\ast~^{m_1}n_1-
~^{\alpha^k_\g(\partial_2(n))}\alpha^k_\m(m')\ast~^{m_1}n_1\\
=&\alpha^k_\m(m)^{\alpha^k_\g(\partial_1(m'))}\ast~^{m_1}n_1-
~^{\alpha^k_\g(\partial_1(m))}\alpha^k_\m(m')\ast~^{m_1}n_1\\
=&\alpha^k_\m(m)^{\partial_1(\alpha^k_\m(m'))}\ast~^{m_1}n_1-
~^{\partial_1(\alpha^k_\m(m))}\alpha^k_\m(m')\ast~^{m_1}n_1\\
=&[\alpha^k_\m(m),\alpha^k_\m(m')]\ast~^{m_1}n_1-
[\alpha^k_\m(m),\alpha^k_\m(m')]\ast~^{m_1}n_1=0.
\end{align*}
The other cases are proved in the same way. Also, It is easily
seen that $\alpha_{\m\ast\n}(\m\square\n)\subseteq\m\square\n$.
Hence, if $\alpha_{\m\square\n}$ is the restriction of
$\alpha_{\m\ast\n}$ to $\m\square\n$, then
$(\m\square\n,\alpha_{\m\square\n})$ is a central ideal of
$(\m\ast\n,\alpha_{\m\ast\n})$.
\begin{definition}
The {\it $($Hom-Leibniz$)$ exterior product}
$(\m\cw\n,\alpha_{\m\cw\n})$ of Hom-Leibniz algebras
$(\m,\alpha_\m)$ and $(\n,\alpha_\n)$ is defined to be the
quotient\vspace{-.15cm}
\[(\m\cw\n,\alpha_{\m\cw\n})=(\frac{\m\ast\n}{\m\square\n},\bar\alpha_{\m\ast
\n}).\vspace{-.15cm}\]
\end{definition}
We write $m\cw n$ and $n\cw m$ to denote the images in $\m\cw\n$
of the generators $m\ast n$ and $n\ast m$, respectively. It is
straightforward to check that the parts $(i)$ and $(ii)$ of
Proposition 3.1 hold with $\ast$ replaced by $\cw$. Moreover, as a
special case of the exterior analogue of Proposition 3.1$(iv)$ we
have that a short exact sequence $e:
(\m,\alpha_\m)\rightarrowtail(\kk,\alpha_\kk)\twoheadrightarrow(\g,\alpha_\g)$
of Hom-Leibniz algebras induces an exact sequence\vspace{-.2cm}
\begin{equation}
(\m\cw\kk,\alpha_{\m\cw\kk})\lo(\kk\cw\kk,\alpha_{\kk\cw\kk})
\twoheadrightarrow(\g\cw\g,\alpha_{\g\cw\g}).\vspace{-.2cm}
\end{equation}
This analogue is obtained from the fact that the images of the
homomorphisms
$(\kk\cw\m,\alpha_{\kk\cw\m})\lo(\kk\cw\kk,\alpha_{\kk\cw\kk})$
and
$(\m\cw\kk,\alpha_{\m\cw\kk})\lo(\kk\cw\kk,\alpha_{\kk\cw\kk})$
are the same. We can say more if the extension $e$ is split.
\begin{proposition}
A split extension
$(\m,\alpha_\m)\rightarrowtail(\kk,\alpha_\kk)\twoheadrightarrow(\g,\alpha_\g)$
of Hom-Leibniz algebras induces an extension of Hom-Leibniz
algebras\vspace{-.2cm}
\[(\m\cw\kk,\alpha_{\m\cw\kk})\stackrel{i}\rightarrowtail(\kk\cw\kk,\alpha_{\kk\cw\kk})
\twoheadrightarrow(\g\cw\g,\alpha_{\g\cw\g}).\vspace{-.2cm}\]
\end{proposition}
\begin{proof}
We only need to prove the injectivity of the homomorphism $i$. We
do this by constructing a homomorphism of Hom-Leibniz algebras
$\psi:(\kk\cw \kk,\alpha_{\kk\cw
\kk})\lo((\m\cw\kk)\rtimes(\g\cw\g),\alpha_{\rtimes})$ such that
$\psi\circ i$ is the canonical inclusion. Here the Hom-Leibniz
action of $(\g\cw\g,\alpha_{\g\cw\g})$ on
$(\m\cw\kk,\alpha_{\m\cw\kk})$ is defined as follows:
\vspace{-.1cm}
\begin{alignat*}{2}
    &~^{x}(m\cw k)=\hspace{-.1cm}[\lambda_\g(x),m]\cw\alpha_\kk(k)
    -[\lambda_\g(x),k]\cw\alpha_\m(m)=-\hspace{-.1cm}~^x(k\cw m),\\
    &~(m\cw
    k)^{x}=[m,\lambda_\g(x)]\cw\alpha_\kk(k)+\alpha_\m(m)\cw[k,\lambda_\g(x)],\\
    &~(k\cw m)^{x}=[k,\lambda_\g(x)]\cw\alpha_\m(m)+\alpha_\kk(k)\cw[m,\lambda_\g(x)],
\end{alignat*}
for all $x\in\g\cw\g$, $m\in\m$, $k\in\kk$, where
$\lambda_\g:\g\cw\g\lo\g$, $(x_1\cw x_2)\longmapsto[x_1,x_2]$, is
the commutator Hom-map and $(\g,\alpha_\g)$ is considered as a
subalgebra of $(\kk,\alpha_\kk)$. Since
$(\kk,\alpha_\kk)\cong(\m\rtimes\g,\alpha_{\rtimes})$, we can
define\vspace{-.15cm}
\begin{alignat*}{1}
\psi:((\m\rtimes\g)\cw(\m\rtimes\g),\alpha_{\cw})&
\lo((\m\cw(\m\rtimes\g))\rtimes(\g\cw\g),\alpha_{\rtimes}).\\
(m_1,x_1)\cw(m_2,x_2)&\longmapsto(m_1\cw(m_2,x_2)-m_2\cw(0,x_1),x_1\cw
x_2)
\end{alignat*}

\vspace{-.2cm}\hspace{-.57cm}The verification shows that $\psi$
preserves the defining relations of the Hom-Leibniz exterior
product with $\psi\circ\alpha_{\cw}=\alpha_{\rtimes}\circ\psi$,
and is hence the required homomorphism.
\end{proof}
To show how the exterior product is related to universal central
extensions of Hom-Leibniz algebras, we need the following
\begin{lemma}
If $e:
(\m,\alpha_\m)\stackrel{\iota}\rightarrowtail(\kk,\alpha_\kk)
\stackrel{\phi}\twoheadrightarrow(\g,\alpha_\g)$ is a central
extension of Hom-Leibniz algebras, then there exists a
homomorphism of Hom-Leibniz algebras
$\psi:(\g\cw\g,\alpha_{\cw})\lo(\kk,\alpha_\kk)$ such that
$\phi\circ\psi(x\cw x')=[x,x']$ for all $x,x'\in\g$. Moreover, if
$(\g,\alpha_\g)$ is perfect then $\psi$ is unique.
\end{lemma}
\begin{proof}
Define the Hom-map $\psi$ on generators by $\psi(x_1\cw
x_2)=[k_1,k_2]$, where $k_i$, $i=1,2$, is any element in pre-image
$x_i$ under $\phi$. Due to the centrality of the extension $e$,
$\psi$ preserves the relations of the exterior product with
$\psi\circ\alpha_{\cw}=\alpha_\kk\circ\psi$, and is thus the
required homomorphism. If
$\psi,\psi':(\g\cw\g,\alpha_{\g\cw\g})\lo(\kk,\alpha_\kk)$ are two
homomorphisms with $\phi\circ\psi=\phi\circ\psi'$, then
$\psi-\psi'=\iota\circ\eta$, where
$\eta:(\g\cw\g,\alpha_{\g\cw\g})\lo(\m,\alpha_\m)$ is a
homomorphism such that $[\g\cw\g,\g\cw\g]$ is contained in
$\ker\eta$. By relation (5a), if $(\g,\alpha_\g)$ is perfect, then
so is $(\g\cw\g,\alpha_{\g\cw\g})$, implying the uniqueness of
$\psi$, as desired.
\end{proof}
The above lemma, together with \cite[Theorem 4.2]{CKR1}, leads us
to the following result.
\begin{proposition}\label{perfect1}
For any perfect Hom-Leibniz algebra $(\g,\alpha_\g)$, the
extension\vspace{-.2cm}
\[(\ker(\lambda_\g),\alpha_{\ker(\lambda_\g)})\rightarrowtail(\g\cw\g,\alpha_{\g\cw\g})
\stackrel{\lambda_\g}\twoheadrightarrow(\g,\alpha_\g)\vspace{-.2cm}\]
is the universal central extension of $(\g,\alpha_\g)$ and so,
there is an isomorphism of Hom-Leibniz algebras
$(\g\cw\g,\alpha_{\g\cw\g})\cong(\g\ast\g,\alpha_{\g\ast\g})$. In
particular,
$HL^{\alpha}_2(\g)\cong(\ker(\lambda_\g),\alpha_{\ker(\lambda_\g)})$.
\end{proposition}
The notion of non-abelian tensor product of Hom-Lie algebras are
introduced and studied in \cite{CKR}. The following fundamental
result presents a relation between the Hom-Lie tensor product and
Hom-Leibniz exterior product of a Hom-Lie algebra.
\begin{theorem}\label{wedten}
Let $(\g,\alpha_\g)$ be a Hom-Lie algebra. Then there is an
isomorphism of Hom-Leibniz algebras
$(\g\cw\g,\alpha_{\g\cw\g})\cong(\g\star\g,\alpha_{\g\star\g})$.
\end{theorem}
\begin{proof}
First we recall that there is a Hom-Leibniz action of
$(\g,\alpha_\g)$ on itself by setting
$^{x}y=[x,y]=-y^x,\vspace{-.1cm}$ for all $x,y\in\g$. Hence, the
Hom-Lie tensor product $\g\star \g$ and the Hom-Leibniz exterior
product $\g\cw \g$ have both the structure of Hom-Leibniz algebra.
By virtue of \cite[Reamak 3]{CKR}, there is a natural epimorphism
of Hom-Leibniz algebras
$\zeta:(\g\ast\g,\alpha_{\g\ast\g})\twoheadrightarrow
(\g\star\g,\alpha_{\g\star\g})$. As
$\g\square\g\subseteq\ker\zeta$, the Hom-map $\zeta$ induces a
surjective homomorphism
$\bar\zeta:(\g\cw\g,\alpha_{\g\cw\g})\to(\g\star\g,\alpha_{\g\star\g})$.
It can be verified that
$\xi:(\g\star\g,\alpha_{\g\star\g})\to(\g\cw\g,\alpha_{\g\cw\g})$,
defined by $\xi(x\star y)=x\cw y$, is a well-defined Hom-map such
that $\xi\circ\bar\zeta=id_{\g\cw\g}$, completing the proof.
\end{proof}
Following \cite{Y}, the $n$-th {\it homology} with trivial
coefficients of a Hom-Lie algebra $(\g,\alpha_\g)$, denoted by
$H_n^\alpha(\g)$, is defined to be the $n$-th homology of the
Hom-chain complex $C_\star^\alpha(\g,\K)$, which is defined
similar to $CL_\ast^\alpha(\g,\K)$, by replacing $\otimes$ by
$\wedge$.  It is routine to show that $HL_1^\alpha(\g)\cong
H_1^\alpha(\g)\cong(\g^{ab},\bar\alpha_{\g^{ab}})$ and there is a
natural epimorphism $\varphi:HL_2^\alpha(\g)\twoheadrightarrow
H_2^\alpha(\g)$. We can use Theorem \ref{wedten} together with
Proposition \ref{perfect1} to show that $\varphi$ is an
isomorphism whenever $(\g,\alpha_\g)$ is perfect.
\begin{corollary}
Let $(\g,\alpha_\g)$ be any perfect Hom-Lie algebra. Then there is
an isomorphism of Hom-Leibniz algebras
$(\g\ast\g,\alpha_{\g\ast\g})\cong(\g\star\g,\alpha_{\g\star\g})$
and so $HL^{\alpha}_2(\g)\cong H^{\alpha}_2(\g)$.
\end{corollary}
\section{The Hopf's formula for Hom-Leibniz algebras}
In \cite{DMK}, Donadze, Martinez and Khmaladze prove that, for any
Leibniz algebra $\g$, the second homology of $\g$, $HL_2(\g)$, is
isomorphic to the kernel of the commutator map
$\g\cw\g\stackrel{[~,~]}\lo\g$ (here $\cw$ denotes the non-abelian
exterior product of Leibniz algebras). Using this result, they
determine the behavior of the functor $HL_2(-)$ with respect to
the direct sum of Leibniz algebras. Also, applying topological
techniques, they get an eight-term exact sequence in homology of
Leibniz algebras\vspace{-.2cm}
\begin{equation}
\hspace{-.2cm}HL_3(\g)\lo HL_3(\kk)\lo HL_2(\g,\m)\lo HL_2(\g)\lo
HL_2(\kk)\lo \m/[\g,\m]\lo HL_1(\g)\twoheadrightarrow
HL_1(\kk).\vspace{-.2cm}
\end{equation}
from a short exact sequence of Leibniz algebras $\m\rightarrowtail
\g\twoheadrightarrow\kk$ (here, $HL_2(\g,\m)$ denotes the second
relative Chevalley-Eilenberg homology of the pair $(\g,\m)$, which
is isomorphic to $\ker(\m\cw\g\lo\g)$). In particular, they obtain
the Hopf's formula for Leibniz algebras and, moreover, show that
if $\ff/\r$ is a free presentation of $\g$ and ${\frak s}/\r$ is
the induced presentation of $\m$ for some ideal ${\frak s}$ of
$\ff$, then $HL_3(\g)\cong\ker(\r\cw\ff\lo\ff)$ and
$HL_3(\kk)\cong\ker({\frak s}\cw\ff\lo\ff)$. In this section, we
generalize these results to Hom-Leibniz algebras. We start with
the following theorem.
\begin{theorem}
Let $(\g,\alpha_\g)$ be any Hom-Leibniz algebra. Then there is an
isomorphism of Hom-vector spaces $ HL_2^\alpha(\g)
\cong(\ker(\lambda_\g),\alpha_{\ker(\lambda_\g)})$, where
$\lambda_\g:(\g\cw\g,\alpha_{\g\cw\g})\lo(\g,\alpha_\g)$ is the
commutator Hom-map.
\end{theorem}
\begin{proof}
We recall that the linear Hom-maps
$d_2:CL_2^\alpha(\g,\K)=(\g^{\otimes2},\alpha_\g^{\otimes2})\to(\g,\alpha_\g)$
and
$d_3:CL_3^\alpha(\g,\K)=(\g^{\otimes3},\alpha_\g^{\otimes3})\to
CL_2^\alpha(\g,\K)=(\g^{\otimes2},\alpha_\g^{\otimes2})$ are
defined by $d_2(x\otimes y)=[x,y]$ and $d_3(x\otimes y\otimes
z)=-[x,y]\otimes\alpha_\g(z)+\alpha_\g(x)\otimes[y,z]+[x,z]\otimes
\alpha_\g(y)$, respectively. It is straightforward to check that
the linear Hom-map
$\psi:((\g\otimes\g)/\Im(d_3),\overline{\alpha_\g^{\otimes2}})
\to(\g\cw\g,\alpha_{\g\cw\g})$ given by $\psi(x\otimes y
+\Im(d_3))=x\cw y$ is well-defined and surjective. We now have the
following diagram of Hom-vector spaces, in which the rows are
exact:

\tikzset{node distance=3cm, auto}$~~~~$
\begin{tikzpicture}[%
back line/.style={densely dotted}, cross
line/.style={preaction={draw=white, -,line width=6pt}}]
\hspace{1.6cm}\node(A1){\fontsize{9.5}{5}\selectfont$~$};
\node[right
of=A1](B1){\fontsize{9.5}{5}\selectfont$HL_2^\alpha(\g)$};
\node[right of=B1](C1){\fontsize{9.5}{5}\selectfont
$(\ds\frac{\g\otimes\g}{{\rm
Im}(d_3)},\overline{\alpha_\g^{\otimes2}})$}; \node[right
of=C1](D1){\fontsize{9.5}{5}\selectfont$([\g,\g],\alpha_{[\g,\g]})$};
\node (B2)[below of=B1, node distance=1.7cm]{\fontsize{9.5}{5}
\selectfont$\hspace{-.3cm}(\ker(\lambda_\g),\alpha_{\ker(\lambda_\g)})$};
\node (C2)[below of=C1, node
distance=1.7cm]{\fontsize{9.5}{5}\selectfont$(\g\cw\g,\alpha_{\g\cw\g})$};
\node (D2)[below of=D1, node
distance=1.7cm]{\fontsize{9.5}{5}\selectfont$([\g,\g],\alpha_{[\g,\g]})$};

\draw[>->](B1) to node{$\subseteq$}(C1); \draw[->>](B1) to
node[right]{\fontsize{9.5}{5}\selectfont$\psi_{|}$}(B2);
\draw[>->](B2) to node{\fontsize{9.5}{5}\selectfont$~$}(C2);
\draw[->>](C1) to
node[right]{\fontsize{9.5}{5}\selectfont$\psi$}(C2);
\draw[->>](C1) to node{\fontsize{9.5}{5}\selectfont$\bar
d_2$}(D1); \draw[->>](C2) to
node{\fontsize{9.5}{5}\selectfont$\lambda_\g$}(D2);
\draw[>->>](D1) to
node[right]{\fontsize{9.5}{5}\selectfont$=$}(D2);
\end{tikzpicture}\\
Note that here we consider $(\g\cw\g,\alpha_{\g\cw\g})$ as a
Hom-vector space and forget the relations (5a,5b,5c,5d) which just
define the Leibniz product. By comparing the relations
(3a,3b,4a,4b,4c) and the generators of $\Im(d_3)$, we conclude
that $\psi$ has an inverse $\psi'$ that sends $x\cw y$ to
$x\otimes y +\Im(d_3)$. It therefore follows that $\psi_|$ is an
isomorphism, as we wished.
\end{proof}
For any Hom-Lie algebra $(\g,\alpha_\g)$, we set
$J_2^{\alpha}(\g)=(\ker(\mu_\g),\alpha_{\ker(\mu_\g)})$, where
$\mu_\g:(\g\star\g,\alpha_{\g\star\g})\lo(\g,\alpha_\g)$ is the
commutator Hom-map of Hom-Lie algebras. The following corollary is
a direct consequence of Theorems 3.6 and 4.1.
\begin{corollary}
For any Hom-Lie algebra $(\g,\alpha_\g)$, $HL_2^{\alpha}(\g)\cong
J_2^{\alpha}(\g)$.
\end{corollary}
\begin{theorem}
Let $(\ff,\alpha_\ff)$ be any free Hom-Leibniz algebra. Then there
is an isomorphism of Hom-Leibniz algebras
$(\ff\cw\ff,\alpha_{\ff\cw\ff})\cong
([\ff,\ff],\alpha_{[\ff,\ff]})$.
\end{theorem}
\begin{proof}
We only require to establish that
$\lambda_\ff:(\ff\cw\ff,\alpha_{\ff\cw\ff})\lo([\ff,\ff],\alpha_{[\ff,\ff]})$
is injective. Using the same notations as in Subsection 2.2,
suppose $(\ff,\alpha_\ff)=(T(\vv)/{\frak
i}_{\vv},\tilde\alpha_{\vv})$ for some Hom-vector space
$(\vv,\alpha_\vv)$. Let $({\frak a},\bar\alpha_{\vv})$ be a
subalgebra of $(T(\vv),\bar\alpha_{\vv})$, where ${\frak a}$ is
the $\bar\alpha_{\vv}$-invariant subspace of $T(\vv)$ generated by
the set $\{ab~|~a,b\in T(\vv)\}$. Note that each element
$x\in{\frak a}$ is written as a unique finite sum
$x=\ds\sum_{i=1}^{n}a_ib_i$, $a_i,b_i\in T(\vv)$. Therefore,
$\phi:({\frak
a},\bar\alpha_{\vv})\lo(\ff\cw\ff,\alpha_{\ff\cw\ff})$, given by
$ab\longmapsto\bar a\cw\bar b$, is a well-defined linear Hom-map
(where $\bar a$ denotes the coset $a+{\frak i}_{\vv}\in
T(\vv)/{\frak i}_{\vv}$). For any $x,y,z\in T(\vv)$, we
have\vspace{-.1cm}
\[\phi(\bar\alpha_{\vv}(x)\cdot(y\cdot z)-(x\cdot y)\cdot
\bar\alpha_{\vv}(z)+(x\cdot z)\cdot\bar\alpha_{\vv}(y))=
\tilde\alpha_{\vv}(\bar x)\cw[\bar y,\bar z]-[\bar x,\bar y]\cw
\tilde\alpha_{\vv}(\bar z)+[\bar x,\bar z]\cw
\tilde\alpha_{\vv}(\bar y)=0.\vspace{-.14cm}\] Hence $\phi({\frak
i}_\vv)=0$ and $\phi$ induces a homomorphism
$\bar\phi:([\ff,\ff]={\frak a}/{\frak
i}_{\vv},\alpha_{[\ff,\ff]})\lo(\ff\cw\ff,\alpha_{\ff\cw\ff})$,
with $\lambda_\ff\circ\bar\phi=id_{[\ff,\ff]}$ and
$\bar\phi\circ\lambda_\ff=id_{\ff\cw\ff}$. This completes the
proof.
\end{proof}
From the above theorem, we deduce the following corollary.
\begin{corollary}
Let $(\r,\alpha_\r)\rightarrowtail(\ff,\alpha_\ff)
\stackrel{\pi}\twoheadrightarrow(\g,\alpha_\g)$ be a free
presentation of the Hom-Leibniz algebra $(\g,\alpha_\g)$. Then
there is an isomorphism of Hom-Leibniz algebras
$(\g\cw\g,\alpha_{\g\cw\g})\cong([\ff,\ff]/[\r,\ff],\bar\alpha_{[\ff,\ff]})$.
In particular, ${\cal
M}(\g,\alpha_\g)\cong(\ker(\lambda_\g),\alpha_{\ker(\lambda_\g)})$.
\end{corollary}
\begin{proof}
Consider the following commutative diagram of Hom-Leibniz algebras
with exacts rows:\vspace{.1cm}

\tikzset{node distance=3cm, auto}$~~~~$
\begin{tikzpicture}[%
back line/.style={densely dotted}, cross
line/.style={preaction={draw=white, -,line width=6pt}}]
\hspace{1.6cm}\node(A1){\fontsize{9.5}{5}\selectfont$~$};
   \node[right of=A1](B1){\fontsize{9.5}{5}\selectfont$(\ff\cw\r,\alpha_{\ff\cw\r})$};
   \node[right of=B1](C1){\fontsize{9.5}{5}\selectfont$(\ff\cw\ff,\alpha_{\ff\cw\ff})$};
   \node[right of=C1](D1){\fontsize{9.5}{5}\selectfont$(\g\cw\g,\alpha_{\g\cw\g})$};
   \node (B2)[below of=B1, node distance=1.7cm]{\fontsize{9.5}{5}\selectfont$(\ker\pi_|,\alpha_{\ker\pi_|})$};
   \node (C2)[below of=C1, node distance=1.7cm]{\fontsize{9.5}{5}\selectfont$([\ff,\ff],\alpha_{[\ff,\ff]})$};
   \node (D2)[below of=D1, node distance=1.7cm]{\fontsize{9.5}{5}\selectfont$([\g,\g],\alpha_{[\g,\g]}),$};

   \draw[->](B1) to node{$\rho$}(C1);
   \draw[->>](B1) to node[right]{\fontsize{9.5}{5}\selectfont${\lambda_\ff|}$}(B2);
   \draw[>->](B2) to node{\fontsize{9.5}{5}\selectfont$~$}(C2);
   \draw[->](C1) to node[right]{\fontsize{9.5}{5}\selectfont$\lambda_\ff$}(C2);
   \draw[->>](C1) to node{\fontsize{9.5}{5}\selectfont$~$}(D1);
   \draw[->>](C2) to node{\fontsize{9.5}{5}\selectfont$\pi_|$}(D2);
   \draw[->>](D1) to node[right]{\fontsize{9.5}{5}\selectfont$\lambda_\g$}(D2);
\end{tikzpicture}\\
Obviously, $\lambda_\ff$ maps the subalgebra
$(\Im(\rho),\alpha_{\ff\cw\ff}|)$ isomorphically onto
$([\r,\ff],\alpha_{[\r,\ff]})$. We therefore conclude from Theorem
4.3 that\vspace{-.2cm}
\[(\g\cw\g,\alpha_{\g\cw\g})\cong(\ds\f{\ff\cw\ff}{\Im(\rho)},\bar\alpha_{\ff\cw\ff})
\cong(\ds\f{[\ff,\ff]}{[\r,\ff]},\bar\alpha_{[\ff,\ff]}),\vspace{-.2cm}\]and
the proof is complete.
\end{proof}
Combining the above corollary with Theorem $4.1$, we get the main
result of this section
\begin{corollary}[Hopf's formula for Hom-Leibniz algebras]
For any Hom-Leibniz algebra $(\g,\alpha_\g)$, there is an
isomorphism of Hom-vector spaces $HL_2^\alpha(\g)\cong{\cal
M}(\g,\alpha_\g)$.
\end{corollary}
As an immediate consequence of Corollary 4.5, we conclude that the
second homology of any free Hom-Leibniz algebra is trivial. Also,
combining the above corollary with Theorem \ref{wedten}, we can
determine the structure of the second homology of abelian Leibinz
algebras.
\begin{corollary}
Let $(\g,\alpha_\g)$ be an abelian Hom-Leibniz algebra. Then there
are isomorphisms\vspace{-.2cm}
$$HL_2^\alpha(\g)\cong (\g\cw\g,\alpha_{\g\cw\g})\cong(\g\star\g,\alpha_{\g\star\g}).\vspace{-.2cm}$$
\end{corollary}
Using Theorem 4.1, we generalize the exact sequence $(3)$ for
Hom-Leibniz algebras.
\begin{theorem}
Let $e: (\m,\alpha_\m)\rightarrowtail(\kk,\alpha_\kk)
\twoheadrightarrow(\g,\alpha_\g)$ be an extension of Hom-Leibniz
algebras. Let $(\r,\alpha_\r)\rightarrowtail(\ff,\alpha_\ff)
\twoheadrightarrow(\kk,\alpha_\kk)$ be a free presentation of
$(\kk,\alpha_\kk)$ and $(\m,\alpha_\m)\cong({\frak
s}/\r,\bar\alpha_{\frak s})$ for some ideal $({\frak
s},\alpha_{\frak s})$ of $(\ff,\alpha_\ff)$. Then there is an
exact sequence of Hom-vector spaces\vspace{-.1cm}
\[
(\ker(\r\cw\ff\lo\ff),\alpha_{\r\cw\ff}|)\lo(\ker({\frak s}\cw
\ff\lo\ff),\alpha_{{\frak s}\cw\ff}|)\lo(\ker(\m\cw\kk\lo
\kk),\alpha_{\m\cw\kk}|)\lo HL_2^{\alpha}(\kk)\vspace{-.08cm}\]
\begin{equation}
\hspace{1.2cm}\lo
HL_2^{\alpha}(\g)\lo(\f{\m}{[\m,\kk]},\bar\alpha_\m)\lo
HL_1^{\alpha}(\kk)\twoheadrightarrow
HL_1^{\alpha}(\g).\vspace{-.1cm}
\end{equation}
Moreover, if the extension $e$ is split, then the sequence $(4)$
induces a short exact sequence\vspace{-.12cm}
\begin{equation}
(\ker(\m\cw\kk\lo\kk),\alpha_{\m\cw\kk}|)\rightarrowtail
HL_2^{\alpha}(\kk)\twoheadrightarrow
HL_2^{\alpha}(\g).\vspace{-.2cm}
\end{equation}
\end{theorem}
\begin{proof}
Consider the following commutative diagrams:\vspace{.3cm}

\tikzset{node distance=2.4cm, auto}$~~~~$
\begin{tikzpicture}[%
back line/.style={densely dotted}, cross
line/.style={preaction={draw=white, -,line width=6pt}}]
\hspace{-2.2cm}\node(A1){\fontsize{9.5}{5}\selectfont$~$};
  \node[right of=A1](B1){\fontsize{9.5}{5}\selectfont$(\m\cw\kk,\alpha_{\m\cw\kk})$};
  \node[right of=B1](C1){\fontsize{9.5}{5}\selectfont$(\kk\cw\kk,\alpha_{\kk\cw\kk})$};
  \node[right of=C1](D1){\fontsize{9.5}{5}\selectfont$(\g\cw\g,\alpha_{\g\cw\g})$};
  \node (B2)[below of=B1, node distance=1.7cm]{\fontsize{9.5}{5}\selectfont$(\m,\alpha_{\m})$};
  \node (C2)[below of=C1, node distance=1.7cm]{\fontsize{9.5}{5}\selectfont$(\kk,\alpha_{\kk})$};
  \node (D2)[below of=D1, node distance=1.7cm]{\fontsize{9.5}{5}\selectfont$(\g,\alpha_{\g})$};

  \draw[->](B1) to node{$~$}(C1);
  \draw[->](B1) to node[right]{\fontsize{9.5}{5}\selectfont$\lambda_\m$}(B2);
  \draw[>->](B2) to node{\fontsize{9.5}{5}\selectfont$~$}(C2);
  \draw[->](C1) to node[right]{\fontsize{9.5}{5}\selectfont$\lambda_\kk$}(C2);
  \draw[->>](C1) to node{\fontsize{9.5}{5}\selectfont$~$}(D1);
  \draw[->>](C2) to node{\fontsize{9.5}{5}\selectfont$~$}(D2);
  \draw[->](D1) to node[right]{\fontsize{9.5}{5}\selectfont$\lambda_\g$}(D2);
\end{tikzpicture}\\\vspace{-2.3cm}

\tikzset{node distance=2.1cm, auto} $~~~~$
\begin{tikzpicture}[%
back line/.style={densely dotted}, cross
line/.style={preaction={draw=white, -,line width=4pt}}]
\hspace{6.7cm} \node(A1){and};

\end{tikzpicture}\\\vspace{-1.96cm}

\tikzset{node distance=2.4cm, auto}$~~~~$
\begin{tikzpicture}[%
back line/.style={densely dotted}, cross
line/.style={preaction={draw=white, -,line width=6pt}}]
\hspace{6.5cm}\node(A1){\fontsize{9.5}{5}\selectfont$~$};
   \node[right of=A1](B1){\fontsize{9.5}{5}\selectfont$(\r\cw\ff,\alpha_{\r\cw\ff})$};
   \node[right of=B1](C1){\fontsize{9.5}{5}\selectfont$({\frak s}\cw\ff,\alpha_{{\frak s}\cw\ff})$};
   \node[right of=C1](D1){\fontsize{9.5}{5}\selectfont$(\m\cw\g,\alpha_{\m\cw\g})$};
   \node (B2)[below of=B1, node distance=1.7cm]{\fontsize{9.5}{5}\selectfont$(\r,\alpha_{\r})$};
   \node (C2)[below of=C1, node distance=1.7cm]{\fontsize{9.5}{5}\selectfont$({\frak s},\alpha_{\frak s})$};
   \node (D2)[below of=D1, node distance=1.7cm]{\fontsize{9.5}{5}\selectfont$(\m,\alpha_{\m})$};

   \draw[->](B1) to node{$~$}(C1);
   \draw[->](B1) to node[right]{\fontsize{9.5}{5}\selectfont$\lambda_\r$}(B2);
   \draw[>->](B2) to node{\fontsize{9.5}{5}\selectfont$~$}(C2);
   \draw[->](C1) to node[right]{\fontsize{9.5}{5}\selectfont$\lambda_{\frak s}$}(C2);
   \draw[->>](C1) to node{\fontsize{9.5}{5}\selectfont$~$}(D1);
   \draw[->>](C2) to node{\fontsize{9.5}{5}\selectfont$~$}(D2);
   \draw[->](D1) to node[right]{\fontsize{9.5}{5}\selectfont$\lambda_\m$}(D2);
\end{tikzpicture}

\hspace{-.65cm}where, the rows are exact. Applying the Snake Lemma
to these diagrams, we get the exact sequences\vspace{.15cm}\\
$(\ker(\lambda_\m),\alpha_{\ker(\lambda_\m)})\stackrel{\delta}\lo(\ker(\lambda_\kk),\alpha_{\ker(\lambda_\kk)})
\lo(\ker(\lambda_\g),\alpha_{\ker(\lambda_\g)})\stackrel{\Delta_1}\lo
(\ds\f{\m}{[\m,\g]},\bar\alpha_\m)\lo
HL_1^{\alpha}(\kk)\twoheadrightarrow
HL_1^{\alpha}(\g)$,\vspace{.15cm}\\
$(\ker(\lambda_\r),\alpha_{\ker(\lambda_\r)})\lo(\ker(\lambda_{\frak
s}),\alpha_{\ker(\lambda_{\frak s})})
\lo(\ker(\lambda_\m),\alpha_{\ker(\lambda_\m)})\stackrel{\Delta_2}\lo
(\ds\f{\r}{[\r,\ff]},\bar\alpha_\r)$,\vspace{.15cm}\\ where
$\Delta_i$, $i=1,2$, is the connecting homomorphism of Hom-vector
spaces. It now remains to show that $\ker\Delta_2=\ker\delta$. But
this immediately follows from the commutative diagram\vspace{.4cm}

\tikzset{node distance=4.cm, auto} $~~~~$
\begin{tikzpicture}[%
back line/.style={densely dotted}, cross
line/.style={preaction={draw=white, -,line width=4pt}}]
\hspace{4cm}
   \node(A1){$(\ker(\lambda_\m),\alpha_{\ker(\lambda_\m)})$};
   \node [right of=A1] (B1) {$\ds(\f{\r}{[\r,\ff]},\bar\alpha_\r)$};
   \node (A2) [below of=A1, node distance=1.6cm]
         {$(\ker(\lambda_\kk),\alpha_{\ker(\lambda_\kk)})$};
   \node [right of=A2] (B2) {$HL_2^{\alpha}(\kk)$,};

   \draw[->](A1) to node {$\delta$} (A2);
   \draw[<-<](B1) to node {$\subseteq$} (B2);
   \draw[->](A1) to node {$\Delta_2$} (B1);
   \draw[>->>](A2) to node {$\beta$} (B2);
\end{tikzpicture}\\
where $\beta$ is the isomorphism obtained in Theorem 4.1. Thus,
the sequence $(4)$ is exact.

We now verify the exactness of $(5)$. According to the points
mentioned at the end of Subsection 2.1, we can assume that
$(\g,\alpha_\g)$ is a subalgebra of $(\kk,\alpha_\kk)$ such that
$\kk=\m\dot{+}\g$. Then \vspace{-.15cm}
\[[\kk,\kk]\cap\m=([\g,\g]+[\m,\kk])\cap
\m=([\g,\g]\cap\m)+[\m,\kk]=[\m,\kk].\vspace{-.15cm}\] Also, owing
to Proposition 3.3, we may consider $(\m\cw\kk,\alpha_{\m\cw\kk})$
as a subalgebra of $(\kk\cw\kk,\alpha_{\kk\cw\kk})$, which implies
that $(\m\cw\kk)\cap\ker(\lambda_\kk)=\ker(\lambda_\m)$. So, we
obtain the following commutative diagram of Hom-Leibniz
algebras:\vspace{.2cm}

\tikzset{node distance=3.2cm, auto}$~~~~$
\begin{tikzpicture}[%
back line/.style={densely dotted}, cross
line/.style={preaction={draw=white, -,line width=6pt}}]
\hspace{1.6cm}\node(A1){\fontsize{9.5}{5}\selectfont$~$};
   \node[right of=A1](B1){\fontsize{9.5}{5}\selectfont$\hspace{-.8cm}(\ker(\lambda_m),\alpha_{\ker(\lambda_m)})$};
   \node[right of=B1](C1){\fontsize{9.5}{5}\selectfont$HL_2^{\alpha}(\kk)$};
   \node[right of=C1](D1){\fontsize{9.5}{5}\selectfont$HL_2^{\alpha}(\g)$};

   \node(B2)[below of=B1, node distance=1.7cm]{\fontsize{9.5}{5}\selectfont
             $\hspace{-1cm}(\m\cw\kk,\alpha_{\m\cw\kk})$};
   \node(C2)[below of=C1, node distance=1.7cm]{\fontsize{9.5}{5}\selectfont
             $(\kk\cw\kk,\alpha_{\kk\cw\kk})$};
   \node(D2)[below of=D1, node distance=1.7cm]{\fontsize{9.5}{5}\selectfont$(\g\cw\g,\alpha_{\g\cw\g})$};

   \node(B3)[below of=B2, node distance=1.7cm]{\fontsize{9.5}{5}\selectfont$\hspace{-1cm}([\m,\kk],\alpha_{[\m,\kk]})$};
   \node(C3)[below of=C2, node distance=1.7cm]{\fontsize{9.5}{5}\selectfont$([\kk,\kk],\alpha_{[\kk,\kk]})$};
   \node(D3)[below of=D2, node distance=1.7cm]{\fontsize{9.5}{5}\selectfont$([\g,\g],\alpha_{[\g,\g]})$,};

   \draw[>->](B1) to node{\fontsize{7.5}{5}\selectfont$~$}(C1);
   \draw[->](C1) to node{\fontsize{7.5}{5}\selectfont$\vartheta|$}(D1);

   \draw[>->](B2) to node{\fontsize{7.5}{5}\selectfont$~$}(C2);
   \draw[->>](C2) to node{\fontsize{7.5}{5}\selectfont$\vartheta$}(D2);

   \draw[>->](B3) to node{\fontsize{7.5}{5}\selectfont$~$}(C3);
   \draw[->>](C3) to node{\fontsize{7.5}{5}\selectfont$~$}(D3);

   \draw[>->](B1) to node[right]{\fontsize{7.5}{5}\selectfont$~$}(B2);
   \draw[>->](C1) to node[right]{\fontsize{7.5}{5}\selectfont$~$}(C2);
   \draw[>->](D1)to node[right]{\fontsize{7.5}{5}\selectfont$~$}(D2);

   \draw[->>](B2) to node[right]{\fontsize{7.5}{5}\selectfont$~$}(B3);
   \draw[->>](C2) to node[right]{\fontsize{7.5}{5}\selectfont$~$}(C3);
   \draw[>->>](D2)to node[right]{\fontsize{7.5}{5}\selectfont$~$}(D3);
\end{tikzpicture}\\
where the columns and rows are exact. One easily sees from the
above diagram that $\vartheta|$ is surjective.
\end{proof}
Using the above theorem and the explanations at the beginning
of this section, we have the following conjecture.\vspace{.15cm}\\
{\bf\large Conjecture.} With the assumptions of Theorem 4.7,
$HL_3^\alpha(\ff)=0$ and there is an isomorphism of Hom-vector
spaces $HL_3^{\alpha}(\g)\cong(\ker({\frak
s}\cw\ff\lo\ff),\alpha_{{\frak s}\cw\ff}|)$.\vspace{.15cm}

Finally, we close this section by examining the behavior of
functors $JL_2^{\alpha}(-)$ and $HL_2^{\alpha}(-)$ to the direct
sum of Hom-Leibniz algebras.
\begin{theorem}
Let $(\g_1,\alpha_{\g_1})$ and $(\g_2,\alpha_{\g_2})$ be two
Hom-Leibniz algebras. Then there are isomorphisms of Hom-vector
spaces\vspace{-.1cm}
\[JL_2^\alpha(\g_1\oplus\g_2)\cong JL_2^\alpha(\g_1)\oplus
JL_2^\alpha(\g_2)\oplus(\g_1^{ab}\ast\g_2^{ab},
\alpha_{\g_1^{ab}\ast\g_2^{ab}})\oplus(\g_2^{ab}\ast\g_1^{ab},
\alpha_{\g_2^{ab}\ast\g_1^{ab}}),\vspace{-.2cm}\]
\[\hspace{-3.5cm}HL_2^{\alpha}(\g_1\oplus\g_2)\cong HL_2^{\alpha}(\g_1)\oplus HL_2^{\alpha}(\g_2)
\oplus(\g_1^{ab}\ast\g_2^{ab},
\alpha_{\g_1^{ab}\ast\g_2^{ab}}).\vspace{-.1cm}\] In particular,
if the endomorphisms $\alpha_{\g_1}$ and $\alpha_{\g_2}$ are
surjective, then\vspace{-.15cm}
\[HL_2^{\alpha}(\g_1\oplus\g_2)\cong HL_2^{\alpha}(\g_1)\oplus HL_2^{\alpha}(\g_2)
\oplus(\g_1^{ab}\otimes\g_2^{ab},
\alpha_{\g_1^{ab}\otimes\g_2^{ab}})\oplus(\g_2^{ab}\otimes\g_1^{ab},
\alpha_{\g_2^{ab}\otimes\g_1^{ab}}).\vspace{-.15cm}\]
\end{theorem}
\begin{proof}
we only need to prove that\vspace{-.1cm}
\begin{equation}
\hspace{.2cm}((\g_1\oplus\g_2)\ast(\g_1\oplus
\g_2),\alpha_{\ast})\cong(\g_1\ast \g_1,\alpha_{\g_1\ast
\g_1})\oplus(\g_2\ast\g_2,\alpha_{\g_2\ast\g_2})\oplus(\g_1^{ab}\ast
\g_2^{ab},\alpha_{\g_1^{ab}\ast \g_2^{ab}})
\oplus(\g_2^{ab}\ast\g_1^{ab},\alpha_{\g_2^{ab}\ast\g_1^{ab}}),\vspace{-.2cm}
\end{equation}
\begin{equation}
\hspace{-3.cm}((\g_1\oplus\g_2)\cw(\g_1\oplus
\g_2),\alpha_{\cw})\cong(\g_1\cw\g_1,\alpha_{\g_1\cw\g_1})\oplus(\g_2\cw\g_2,
\alpha_{\g_2\cw\g_2})\oplus(\g_1^{ab}\ast\g_2^{ab},\alpha_{\g_1^{ab}\ast\g_2^{ab}}).\vspace{-.1cm}
\end{equation}
where $\alpha_\ast$ and $\alpha_\cw$ denote the endomorphisms of
$(\g_1\oplus\g_2)\ast(\g_1\oplus \g_2)$ and
$(\g_1\oplus\g_2)\cw(\g_1\oplus \g_2)$ induced by $\alpha_{\g_1}$
and $\alpha_{\g_2}$, respectively. Put
$(\g,\alpha_\g)=(\g_1\oplus\g_2,\alpha_{\oplus})$, and identify
$(\g_1,\alpha_{\g_1})$ and $(\g_2,\alpha_{\g_2})$ with their
images in $(\g,\alpha_\g)$. Then $[\g_1,\g_2]=0$. We claim that
for any ideal $(\kk,\alpha_\kk)$ of $(\g,\alpha_\g)$, there is an
isomorphism of Hom-Leibniz algebras\vspace{-.2cm}
\begin{equation}
(\kk\ast\g,\alpha_{\kk\ast\g})\cong((\kk\ast\g_1)\oplus(\kk\ast\g_2),
\alpha_{\oplus}).\vspace{-.2cm}
\end{equation}
Define the Hom-map
$\varphi:(\kk\ast\g,\alpha_{\kk\ast\g})\lo((\kk\ast\g_1)\oplus(\kk
\ast\g_2),\alpha_{\oplus})$, $(k\ast(x_1,x_2))\longmapsto((k\ast
x_1),(k\ast x_2))$ and $((x_1,x_2)\ast k)\longmapsto((x_1\ast
k),(x_2\ast k))$. A page of routine calculations shows that
$\varphi$ preserves the defining relations of the exterior product
and also,
$\varphi\circ\alpha_{\kk\ast\g}=\alpha_{\oplus}\circ\varphi$. On
the other hand, the inclusions of $\g_1$ and $\g_2$ into $\g$
yield linear Hom-maps $(\kk\ast\g_1,\alpha_{\kk\ast\g_1})\lo
(\kk\ast\g,\alpha_{\kk\ast\g})$ and
$(\kk\ast\g_2,\alpha_{\kk\ast\g_2})\lo(\kk\ast\g,\alpha_{\kk\ast\g})$,
which combine to give an inverse $\psi$ of $\varphi$. It therefore
follows that $\varphi$ is an isomorphism of Hom-Leibniz algebras,
as claimed. We now get the isomorphism $(6)$ by applying the
isomorphism $(8)$ twice.

For the isomorphism $(7)$, it is enough to note that for the
isomorphism $\varphi$ obtained in proving $(6)$, we have
$\varphi((\g_1\oplus\g_2)\square(\g_1\oplus\g_2))=(\g_1\square
\g_1)\oplus(\g_2\square\g_2)$, which deduces the result.
\end{proof}
The following corollary generalizes a result due to Ellis
\cite[Proposition 21]{E}.
\begin{corollary}
Let $(\g_1,\alpha_{\g_1})$ and $(\g_2,\alpha_{\g_2})$ be two
Hom-Lie algebras with surjective endomorphisms $\alpha_{\g_1}$ and
$\alpha_{\g_2}$. Then there is an isomorphism of Hom-vector
spaces\vspace{-.15cm}
\[J_2^{\alpha}(\g_1\oplus\g_2)\cong J_2^{\alpha}(\g_1)\oplus J_2^{\alpha}(\g_2)
\oplus(\g_1^{ab}\otimes\g_2^{ab},
\alpha_{\g_1^{ab}\otimes\g_2^{ab}})\oplus(\g_2^{ab}\otimes\g_1^{ab},
\alpha_{\g_2^{ab}\otimes\g_1^{ab}}).\vspace{-.15cm}\]
\end{corollary}
\begin{proof}
It follows immediately from Theorem 4.8 and Corollary 4.2.
\end{proof}
\section{Capable Hom-Leibniz algebras}
In this section we present a powerful property to classify
Hom-Leibniz algebras. This is due to the fact that the capability
property of algebraic structures is of fundamental importance in
classification of Lie and Leibniz algebras. Also, the study of
capability includes a link between the homology theory of algebras
and the tensor and the exterior centers. Investigating this
concept in Hom-Leibniz algebras can be more challenging since the
center of these algebras does not need to be an ideal. In choosing
a starting point, it seems reasonable to employ the most natural
way to define this concept.
\begin{definition}
 A Hom-Leibniz algebra $(\g,\alpha_\g)$ is
said to be {\it capable} if there exists a Hom-Leibniz algebra
$(\kk,\alpha_\kk)$ such that
$(\g,\alpha_\g)\cong(\kk/Z_\alpha(\kk),\bar\alpha_\kk)$.
\end{definition}
When $\alpha_\g=id_\g$, this definition recovers the notion of
capable Leibniz algebra in \cite{HES,KKL}.
\begin{example}
$(i)$ Let $(\g,\alpha_\g)$ be an abelian Hom-Leibniz algebra with
linear basis $\{e_i~|~i\in I\}$, where $I$ is a non-empty set.
Consider the Hom-Leibniz algebra $(\kk,\alpha_\kk)$ with linear
basis $\{e_i, e_{jk}~|~i,j,k\in I\}$, the product given by
$[e_j,e_k]=e_{jk}$ for all $j,k\in I$, and zero elsewhere, and the
endomorphism $\alpha_\kk$ is defined as
$\alpha_\kk(e_i)=\alpha_\g(e_i)$ and
$\alpha_\kk(e_{jk})=[\alpha_\kk(e_j),\alpha_\kk(e_k)]$. Since all
triple products of elements in $\kk$ are zero, one sees that
$(\kk,\alpha_\kk)$ satisfies the condition of  Hom-Leibniz
identity. Furthermore, $Z_{\alpha}(\kk)=span\{e_{jk}~|~j,k\in
I\}$, and
$(Z_{\alpha}(\kk),\alpha_{Z_{\alpha}(\kk)})\rightarrowtail(\kk,\alpha_\kk)
\stackrel{\pi}{\twoheadrightarrow}(\g,\alpha_\g)$ is an extension
of Hom-Leibniz algebras, where $\pi$ is defined by $\pi(e_i)=e_i$
and $\pi(e_{jk})=0$. Hence $(\g,\alpha_\g)$ is capable.

$(ii)$ Let $(\g,\alpha_\g)$ be a Hom-Leibniz algebra with
$\alpha_\g=0$. We claim that $(\g,\alpha_\g)$ is capable. The
result is clear when $Z_{\alpha}(\g)=0$ or $(\g,\alpha_\g)$ is
abelian. Choose a linear basis $\{e_i~|~i\in I\}$ for
$Z_{\alpha}(\g)$ and extend it to a linear basis $\{e_i,
f_j~|~i\in I,j\in J\}$ for $\g$, where $I$ and $J$ are non-empty
sets. Take the vector space $\kk$ with a linear basis $\{e_i, t_i,
f_j~|~i\in I, j\in J\}$ (in fact, $\kk$ is a vector space direct
sum of $\g$ and the space generated by the set $\{t_i~|~i\in
I\}$), together with the following product: $[e_i,f_{j_0}]=t_i$
for some fixed $j_0\in J$ and all $i\in I$, $[f_j,f_k]$ is the
same in $\g$ for all $j,k\in J$, and zero elsewhere. Then $\kk$
with companion endomorphism $\alpha_\kk=0$ is a Hom-Leibniz
algebra such that $Z_\alpha(\kk)=span\{t_i~|~i\in I\}$ and
$(\kk/Z_{\alpha}(\kk),\bar\alpha_\kk)\cong(\g,\alpha_\g)$, as
claimed.

$(iii)$ Consider the three-dimensional Hom-Leibniz algebra
$(\g,\alpha_\g)$ with linear basis $\{e_1,e_2,e_3\}$, the product
given by $[e_1,e_2]=e_3$ (unwritten products are equal to zero),
and the non-surjective endomorphism $\alpha_\g$ defined by
$\alpha_\g(e_1)=e_3$, $\alpha_\g(e_2)=e_2$, $\alpha_\g(e_3)=0$.
Evidently, $Z_\alpha(\g)=\g^2=span\{e_3\}$. Let $(\kk,\alpha_\kk)$
be the four-dimensional Hom-Leibniz algebra with linear basis
$\{f_1,f_2,f_3,f_4\}$, the product given by $[f_1,f_2]=f_3$,
$[f_3,f_1]=f_4$, and the endomorphism defined by
$\alpha_\kk(f_1)=f_3$, $\alpha_\kk(f_2)=f_2$,
$\alpha_\kk(f_3)=\alpha_\kk(f_4)=0$. Then
$Z_\alpha(\kk)=span\{f_4\}$ and
$(\kk/Z_{\alpha}(\kk),\bar\alpha_\kk)\cong(\g,\alpha_\g)$, that
is, $(\g,\alpha_\g)$ is capable.  (Note that we give the general
case of this example in Theorem \ref{nonperfect} below).
\end{example}
The following concepts are useful in the study of capability of
Hom-Leibniz algebras.
\begin{definition}
 Let $(\g,\alpha_\g)$ be a Hom-Leibniz
algebra.

$(a)$ The {\it tensor center} $Z^{\ast}_{\alpha}(\g)$ of
$(\g,\alpha_\g)$ is defined to be the vector space \vspace{-.2cm}
\[Z^{\ast}_{\alpha}(\g)=\{x\in L~|~\alpha^k_\g(x)\ast y=y\ast
\alpha^k_\g(x)=0_{\g\ast\g},~{\rm for~all}~y\in\g,
k\geq0\}.\vspace{-.2cm}\]

$(b)$ The {\it exterior center} $Z^{\cw}_{\alpha}(\g)$ of
$(L,\alpha_L)$ is defined to be\vspace{-.2cm}
\[Z^{\cw}_{\alpha}(\g)=\{x\in\g~|~\alpha^k_\g(x)\cw y=y\cw
\alpha^k_\g(x)=0_{\g\cw\g},~{\rm for~all}~y\in\g,
k\geq0\}.\vspace{-.2cm}\]
\end{definition}
Plainly, $Z^{\ast}_{\alpha}(\g)\subseteq Z^{\cw}_{\alpha}(\g)$,
and the equality holds whenever $(\g,\alpha_\g)$ is perfect, by
Proposition 3.5.

The following proposition provides some useful information about
the tensor and exterior centers.
\begin{proposition}
Let $(\g,\alpha_\g)$ be a Hom-Leibniz algebra. Then:

$(i)$ Both
$(Z^{\ast}_{\alpha}(\g),\alpha_{Z^{\ast}_{\alpha}(\g)})$ and
$(Z^{\cw}_{\alpha}(\g),\alpha_{Z^{\cw}_{\alpha}(\g)})$ are central
ideals of $(\g,\alpha_\g)$.

$(ii)$ If $\alpha_\g$ is a surjective endomorphism, then
$Z^{\ast}_{\alpha}(\g)$ and $Z^{\cw}_{\alpha}(\g)$ is contained in
$[\g,\g]$.

$(iii)$ If
$(\r,\alpha_\r)\rightarrowtail(\ff,\alpha_\ff)\stackrel{\pi}
{\twoheadrightarrow}(\g,\alpha_\g)$ is a free presentation of
$(\g,\alpha_\g)$, then
$Z^{\cw}_{\alpha}(\g)=\bar\pi(Z_{\alpha}(\ff/[\r,\ff]))$, where
$\bar{\pi}:(\ff/[\r,\ff],\bar{\alpha}_\ff)\lo(\g,\alpha_\g)$ is
the homomorphism induced by $\pi$.

$(iv)$ $(Z^{\cw}_{\alpha}(\g),\alpha_{Z^{\cw}_{\alpha}(\g)})$ is
the smallest central subalgebra containing all central subalgebras
$(\n,\alpha_\n)$ for which the canonical homomorphism
$HL_2^{\alpha}(\g)\lo HL_2^{\alpha}(\g/\n)$ is a monomorphism
$($or equivalently, for which the canonical homomorphism
$(\g\cw\g,\alpha_{\g\cw\g})\lo(\g/\n\cw\g/\n,\alpha_{\g/\n\cw\g/\n})$
is an isomorphism$)$.

$(v)$ $(Z^{\ast}_{\alpha}(\g),\alpha_{Z^{\ast}_{\alpha}(\g)})$ is
the smallest central subalgebra containing all central subalgebras
$(\n,\alpha_\n)$ for which the canonical homomorphism
$JL_2^{\alpha}(\g)\lo JL_2^{\alpha}(\g/\n)$ is a monomorphism
$($or equivalently, for which the canonical homomorphism
$(\g\ast\g,\alpha_{\g\ast\g})\lo(\g/\n\ast\g/\n,\alpha_{\g/\n\ast\g/\n})$
is an isomorphism$)$.
\end{proposition}
\begin{proof}
$(i)$ Applying the first six identities in the definition of the
tensor product, it follows that $Z^{\ast}_{\alpha}(\g)$ is a
vector subspace of $\g$. Considering the commutator Hom-map
$\lambda_\g:(\g\ast\g,\alpha_{\g\ast\g})\lo(\g,\alpha_\g)$, we
have\vspace{-.2cm}
\[[\alpha^k_\g(x),y]=\lambda_\g(\alpha^k_\g(x)\ast y)=0~~~~{\rm and}~~~~
[y,\alpha^k_\g(x)]=\lambda_\g(y\ast\alpha^k_\g(x))=0,\vspace{-.15cm}\]
for all $x\in Z^{\ast}_{\alpha}(\g)$, $y\in\g$ and $k\geq0$. Then
$Z^{\ast}_{\alpha}(\g)\subseteq Z(\g)$, implying that
$(Z^{\ast}_{\alpha}(\g),\alpha_{Z^{\ast}_{\alpha}(\g)})$ is an
ideal. The same arguments are used for the exterior center.

$(ii)$ Consider the composite homomorphism\vspace{-.2cm}
\[(\g\ast\g,\alpha_{\g\ast\g})\stackrel{\rho}{\lo}
(\g^{ab}\ast\g^{ab},\alpha_{\g^{ab}\ast\g^{ab}})\stackrel{\phi}{\lo}((\g^{ab}\otimes
\g^{ab})\oplus(\g^{ab}\otimes\g^{ab}),\alpha_{\oplus}),\vspace{-.2cm}\]
where $\rho$ is the natural surjective homomorphism and $\phi$ is
the isomorphism given in Proposition 3.1$(iii)$. If there is some
$x\in\g-[\g,\g]$, then $\bar x\otimes\bar y\neq 0$ for some
$y\in\g$, implying that the pre-image of this element under $\phi
\circ\rho$ can not be vanished, that is, $x\ast y\neq 0$ in
$\g\ast\g$. This means that $Z^\ast_\alpha(\g)\subseteq[\g,\g]$.\\
For second part, we consider the composite
homomorphism\vspace{-.2cm}
\[(\g\cw \g,\alpha_{\g\cw \g})\stackrel{\rho'}{\lo}
(\g^{ab}\cw\g^{ab},\alpha_{\g^{ab}\cw\g^{ab}})\stackrel{\bar\zeta}{\lo}
(\g^{ab}\star\g^{ab},\alpha_{\g^{ab}\star\g^{ab}})\stackrel{\phi'}{\lo}
(\g^{ab}\otimes\g^{ab},\alpha_{\g^{ab}\otimes\g^{ab}}),\vspace{-.2cm}\]
where $\rho'$ is the natural surjective homomorphism, $\bar\zeta$
and $\phi'$ are the isomorphisms given in Theorem 3.6 and
\cite[Proposition 2.1]{CKR}, respectively. The reminder of proof
is similar to above.

$(iii)$ Assume $\bar x\in Z_{\alpha}(\ff/[\r,\ff])$, that is,
$[\bar\alpha^k_\ff(\bar x),\bar f]=[\bar f,\bar\alpha^k_{\ff}(\bar
x)]=0$ for all $\bar f\in \ff/[\r,\ff]$ and $k\geq0$. It follows
from the isomorphism given in Corollary 4.4 that\vspace{-.2cm}
\[\alpha^k_\g(\bar\pi(\bar x))\cw\bar\pi(\bar f)
=\bar\pi(\bar\alpha^k_\ff(\bar x))\cw\bar\pi(\bar f)=0~~~{\rm
and}~~~\bar\pi(\bar f)\cw\alpha^k_\g(\bar\pi(\bar x))=\bar\pi(\bar
f)\cw\bar\pi(\bar\alpha^k_\ff(\bar x))=0.\vspace{-.2cm}\]Hence,
the subjectivity of $\bar\pi$ yields that $\bar\pi(\bar x)\in
Z^{\cw}_{\alpha}(\g)$. The reverse containment is proved
similarly.

$(iv)$ The result directly follows from the following commutative
diagram of Hom-Leibniz algebras\vspace{.1cm}

\tikzset{node distance=3cm, auto}$~~~~$
\begin{tikzpicture}[%
back line/.style={densely dotted}, cross
line/.style={preaction={draw=white, -,line width=6pt}}]
\hspace{1cm}\node(A1){\fontsize{9.5}{5}\selectfont$~$};
\node[right
of=A1](B1){\fontsize{9.5}{5}\selectfont$\hspace{.cm}(\n\cw\g,\alpha_{\n\cw\g})$};
\node[right
of=B1](C1){\fontsize{9.5}{5}\selectfont$HL_2^{\alpha}(\g)$};
\node[right
of=C1](D1){\fontsize{9.5}{5}\selectfont$HL_2^{\alpha}(\ds\f{\g}{\n})$};

   \node(B2)[below of=B1, node distance=1.5cm]{\fontsize{9.5}{5}\selectfont $\hspace{.cm}(\n\cw\g,\alpha_{\n\cw\g})$};
   \node(C2)[below of=C1, node distance=1.5cm]{\fontsize{9.5}{5}\selectfont $(\g\cw\g,\alpha_{\g\cw\g})$};
             \node(D2)[below of=D1, node distance=1.5cm]
            {\fontsize{9.5}{5}\selectfont$(\ds\f{\g}{\n}\cw\f{\g}{\n},\alpha_{\f{\g}{\n}\cw\f{\g}{\n}})$};

   \node(B3)[below of=B2, node distance=1.5cm]{\fontsize{9.5}{5}\selectfont $~$};
   \node(C3)[below of=C2, node distance=1.5cm]{\fontsize{9.5}{5}\selectfont
            $([\g,\g],\alpha_{[\g,\g]})$};
   \node(D3)[below of=D2, node distance=1.5cm]
     {\fontsize{9.5}{5}\selectfont$(\ds[\f{\g}{\n},\f{\g}{\n}],\alpha_{[\f{\g}{\n},\f{\g}{\n}]})$,};

   \draw[->](B1) to node{\fontsize{7.5}{5}\selectfont$~$}(C1);
   \draw[->](C1) to node{\fontsize{7.5}{5}\selectfont$~$}(D1);

   \draw[->](B2) to node{\fontsize{7.5}{5}\selectfont$~$}(C2);
   \draw[->>](C2) to node{\fontsize{7.5}{5}\selectfont$~$}(D2);

  \draw[->>](C3) to node{\fontsize{7.5}{5}\selectfont$~$}(D3);

  \draw[>->>](B1) to node[right]{\fontsize{7.5}{5}\selectfont$=$}(B2);
  \draw[>->](C1) to node[right]{\fontsize{7.5}{5}\selectfont$~$}(C2);
  \draw[>->](D1)to node[right]{\fontsize{7.5}{5}\selectfont$~$}(D2);

\draw[->>](C2) to
node[right]{\fontsize{7.5}{5}\selectfont$~$}(C3); \draw[->>](D2)to
node[right]{\fontsize{7.5}{5}\selectfont$~$}(D3);
\end{tikzpicture}\\
where, by the sequence $(2)$ and Theorem 4.1, the rows and columns
are exact.

The proof of $(v)$ is similar to that of $(iv)$.
\end{proof}
We have the following interesting consequence.
\begin{corollary}
Let $(\n,\alpha_ n)$ be a central subalgebra of a finite
dimensional Hom-Leibniz algebra $(\g,\alpha_\g)$. Then
$\n\subseteq Z^{\cw}_{\alpha}(\g)$ if and only if
$\dim(HL_2^{\alpha}(\g/\n))=\dim(HL_2^{\alpha}(\g))+\dim(\n\cap[\g,\g])$
$($as vector spaces$)$.
\end{corollary}
\begin{proof}
The centrality of $(\n,\alpha_\n)$ together with Theorem 4.7
implies an exact sequence\vspace{-.2cm}
\[(\n\cw\kk,\alpha_{\n\cw\kk})\lo HL_2^{\alpha}(\f{\g}{\n})\lo
HL_2^{\alpha}(\g)\twoheadrightarrow(\n\cap\g,\alpha_{\n\cap
\g}).\vspace{-.2cm}\]The result now follows from Proposition
5.4$(iv)$.
\end{proof}
If $\alpha_\g=id_\g$, Corollary 5.5 reduce to \cite[Corollary
2.1]{HES}.

The following theorem provides a criterion for determining the
capability of Hom-Leibniz algebra, which is similar to work of
Casas and Martinez \cite[Corollary 4.10]{CG} for the Hom-Lie
algebra case.
\begin{theorem}
A Hom-Leibniz algebra $(\g,\alpha_\g)$ is capable if and only if
$Z^{\cw}_{\alpha}(\g)=0$.
\end{theorem}
\begin{proof}
Consider the central free presentation\vspace{-.12cm}
\begin{equation}
e_1:~(\frac{\r}{[\r,\ff]},\bar\alpha_\r)\rightarrowtail(\frac{\ff}{[\r,\ff]},\bar\alpha_\ff)
\stackrel{\bar\pi}\twoheadrightarrow(\g,\alpha_\g)\vspace{-.12cm}
\end{equation}
of $(\g,\alpha_\g)$ induced by its free presentation. By virtue of
Proposition 5.4$(iii)$, it suffices to show that\vspace{-.18cm}
\[(\g,\alpha_\g) {\rm~is~capable~if~and~only~if~}
Z_\alpha(\ff/[\r,\ff])=\r/[\r,\ff].\vspace{-.18cm}\]Assume first
that $(\g,\alpha_\g)$ is capable. Then one finds a central
extension\vspace{-.15cm}
\[(Z_\alpha(\kk_1),\alpha_{Z_\alpha(\kk_1)})\rightarrowtail(\kk_1,\alpha_{\kk_1})
\stackrel{\rho}\twoheadrightarrow(\g,\alpha_\g).\vspace{-.15cm}\]
By the freeness of $(\ff,\alpha_\ff)$, there is a homomorphism
$\beta:(\ff/[\r,\ff],\bar\alpha_\ff)\lo(\kk_1,\alpha_{\kk_1})$
such that $\rho\circ\beta=\bar\pi$. It is routine to see that
$\kk=\m+\Im(\beta)$ and then
$\bar\pi(Z_\alpha(\ff/[\r,\ff]))=\rho(Z_\alpha(\kk))=0$, forcing
$Z(\ff/[\r,\ff])=\r/[\r,\ff]$. The proof of the converse is
trivial.
\end{proof}
We are able to obtain the following important result, which is a
vast generalization of a result stated in Khmaladze et al.
\cite{KKL}.
\begin{corollary}
Let $(\g,\alpha_\g)$ be a perfect Hom-Leibniz algebra with
surjective endomorphism $\alpha_\g$. Then $(\g,\alpha_\g)$ is
capable if and only if $Z(\g)=0$.
\end{corollary}
\begin{proof}
Note that the surjectivity of $\alpha_\g$ implies that
$Z_\alpha(\g)=Z(\g)$. By view of Proposition 3.1$(iii)$,
$(Z(\g)\ast\g,\alpha_{Z(\g)\ast\g})\cong
((Z(\g)\otimes\g^{ab})\oplus(\g^{ab}\otimes
Z(\g)),\alpha_{\oplus})=0$. Hence
$(Z(\g)\cw\g,\alpha_{Z(\g)\cw\g})=0$, implying that the natural
homomorphism $HL_2^{\alpha}(\g)\lo HL_2^{\alpha}(\g/Z(\g))$ is
injective. Therefore, applying Proposition 5.4$(iv)$ and Theorem
5.6, we have the result.
\end{proof}
In the following, we try to omit the condition of surjectivity in
Corollary 5.8.
\begin{proposition}
Let $(\g,\alpha_\g)$ be any perfect Hom-Leibniz algebra. Then
$(\g,\alpha_\g)$ is capable if and only if
$Z_{\alpha}(\g)\subseteq\ker(\alpha_\g)$.
\end{proposition}
\begin{proof}
Without loss of generality, we can assume $Z_\alpha(\g)\neq0$. If
there is $x\in Z_{\alpha}(\g)$ such that $\alpha_\g(x)\neq0$,
then, applying the defining relations of the exterior product,
$\alpha_\g^k(x)\cw[a,b]=[a,b]\cw\alpha_\g^k(x)=0$ for all
$a,b\in\g$ and $k\geq 1$. The perfectness of $(\g,\alpha_\g)$
yields that $\alpha_\g(x)\in Z_\alpha^\cw(\g)$, and then
$(\g,\alpha_\g)$ is not capable, thanks to Theorem 5.6. We now
suppose $\alpha_\g(Z_{\alpha}(\g))=0$ and consider the algebra
$\kk$ introduced in Example 5.2$(ii)$. We define the endomorphism
$\alpha_\kk$ such that $\alpha_\kk|_{\g}=\alpha_\g$ and
$\alpha_\kk|_{span\{t_i|i\in I\}}=0$. It is straightforward to
check that $(\kk,\alpha_\kk)$ is a Hom-Leibniz algebra with
$Z_{\alpha}(\g)=span\{t_i|i\in I\}$ and
$(\kk/Z_{\alpha}(\kk),\bar\alpha_\kk)\cong(\g,\alpha_\g)$.
\end{proof}
The following theorem shows that for deciding on the property of
capability, we may restrict ourselves to Hom-Leibniz algebras with
surjective endomorphisms.
\begin{theorem}\label{nonperfect}
Let $(\g,\alpha_\g)$ be a non-perfect Hom-Leibniz algebra with
non-surjective endomorphism $\alpha_\g$. Then $(\g,\alpha_\g)$ is
capable.
\end{theorem}
\begin{proof}
We first assume that $\g-[\g,\g]\subseteq\Im(\alpha_\g)$, and
choose a fixed element $x\in\g-[\g,\g]$. Then for any
$y\in[\g,\g]$, there are $x_1,x_2\in\g$ such that
$\alpha_\g(x_1)=x$ and $\alpha_\g(x_2)=x+y$, implying that
$\alpha_\g(x_2-x_1)=y$. It therefore follows that
$[\g,\g]\subseteq\Im(\alpha_\g)$ and then $\alpha_\g$ is
surjective, a contradiction. So, we can find an element $x\in
\g-([\g,\g]\cup\Im(\alpha_\g))$. Choose a linear basis
$\{e_i~|~i\in I\}$ for $Z_{\alpha}(\g)$ and and extend it to a
linear basis $\{e_i, f_j~|~i\in I,j\in J\}$ for $\g$ containing
$x$, where $I$ and $J$ are non-empty sets. Take the vector space
$\kk=\g\dot{+}\frak{a}$ such that $\frak a$ is generated by the
set $\{t_i|~i\in I\}$. Consider the following product in $\kk$:
$[e_i,x]=t_i$ for all $i\in I$, $[f_j,f_k]$ is the same in $\g$
for all $j,k\in J$, and zero elsewhere. Define the endomorphism
$\alpha_\kk$ of $\kk$ such that $\alpha_\kk|_\g=\alpha_\g$ and
$\alpha_\kk(t_i)=0$ for $i\in I$. Since
$x\not\in\Im(\alpha_\kk)\cup [\kk,\kk]$, $x$ does not appear in
the brackets of the forms $[\alpha(a),b]$ or $[b,\alpha(a)]$,
where $a\in \kk,b\in[\kk,\kk]$. Hence for any $i\in I$, the
elements $t_i$ are never vanished by the Hom-Leibniz identity.
Also, for any $i\in I$ we have $[t_i,\kk]=[\kk,t_i]=0$. Therefore,
$(\kk,\alpha_\kk)$ is a Hom-Leibniz algebra with
$Z_{\alpha}(\kk)=\frak a$ and, moreover,
$(\kk/Z_{\alpha}(\kk),\bar\alpha_\kk)\cong(\g,\alpha_\g)$, as
desired.
\end{proof}
It is obvious that if the Hom-Leibniz algebras
$(\g_1,\alpha_{\g_1})$ and $(\g_2,\alpha_{\g_2})$ are capable,
then the direct sum $(\g_1\oplus\g_2,\alpha_{\g_1\oplus\g_2})$ is
capable, because $Z^{\cw}_{\alpha}(\g_1\oplus\g_2)\subseteq
Z^{\cw}_{\alpha}(\g_1)\oplus Z^{\cw}_{\alpha}(\g_2)$. In the
following, we prove the converse of this result, under some
conditions.
\begin{theorem}
Let $(\g,\alpha_\g)$ be a finite dimensional capable regular
Hom-Leibniz algebra such that
$(\g,\alpha_\g)=(\g_1\oplus\g_2,\alpha_{\g_1\oplus\g_2})$. Then
$(\g_1,\alpha_{\g_1})$ and $(\g_2,\alpha_{\g_2})$ are capable.
\end{theorem}
\begin{proof}
We only need to prove that
$Z^{\cw}_{\alpha}(\g_1\oplus\g_2)=Z^{\cw}_{\alpha}(\g_1)\oplus
Z^{\cw}_{\alpha}(\g_2)$. By Example 5.1$(i)$, we can assume that
$(\g,\alpha_\g)$ is non-abelian. For convenience, we divide the
rest of the proof into three steps.

{\it Step} $1$. Here we prove that if $(\g_i,\alpha_{\g_i})$ is
non-abelian, then $(\g_i,\alpha_{\g_i})=(\T_i\oplus{\frak
a}_i,\alpha_{\T_i\oplus{\frak a}_i})$ such that $({\frak
a}_i,\alpha_{{\frak a}_i})$ is an abelian Hom-Leibniz algebra and
$Z^{\cw}_{\alpha}(\g_i)=Z^{\cw}_{\alpha}(\T_i)$, for $i=1,2$.

By arguments similar to those used in \cite[Theorem 12]{M} we can
deduce that $(\g_i,\alpha_{\g_i})=(\T_i\oplus{\frak
a}_i,\alpha_{\oplus})$ in which $({\frak a}_i,\alpha_{{\frak
a}_i})$ is abelian and $[\g_i,\g_i]\cap
Z(\g_i)=Z(\T_i)\subseteq[\T_i,\T_i]$ (see also \cite[Theorem
3.8]{PNP} for the case of Hom-Lie algebras). By Example 5.1$(i)$,
$({\frak a}_i,\alpha_{{\frak a}_i})$ is capable and then
$Z^{\cw}_{\alpha}({\frak a}_i)=0$, implying that
$Z^{\cw}_{\alpha}(\g_i)\subseteq Z^{\cw}_{\alpha}(\T_i)$. We now
claim that $Z^{\cw}_{\alpha}(\T_i)\subseteq
Z^{\cw}_{\alpha}(\g_i)$. By virtue of Theorem 4.8 and Corollary
5.5, we have\vspace{-.2cm}
\begin{alignat*}{1}
\dim(HL_2^{\alpha}(\g_i))&=\dim(H_2^{\alpha}(\T_i))+
\dim(HL_2^{\alpha}({\frak a}_i))+\dim(((\T_i^{ab}\otimes{\frak
a}_i)\oplus
({\frak a}_i\otimes\T_i^{ab}),\alpha_{\oplus}))\\
&=\dim(HL_2^{\alpha}(\f{\T_i}{Z^{\cw}_{\alpha}(\T_i)}))-\dim(Z^{\cw}_{\alpha}(\T_i))+
\dim(HL_2^{\alpha}({\frak a}_i))\\
&+\dim(((\T_i^{ab}\otimes{\frak a}_i)\oplus
({\frak a}_i\otimes\T_i^{ab}),\alpha_{\oplus}))\\
&=\dim(HL_2^{\alpha}(\f{\g_i}{Z^{\cw}_{\alpha}(\T_i)}))-\dim(Z^{\cw}_{\alpha}(\T_i))
\end{alignat*}
which, using again Corollary 5.5, gives the required result.

{\it Step} $2$. Here we prove that if the Hom-Leibniz algebras
$(\g_i,\alpha_{\g_i})$, $i=1,2$, are both non-abelian and
$Z(\g_i)\subseteq[\g_i,\g_i]$, then $Z^{\cw}_{\alpha}(\g)=
Z^{\cw}_{\alpha}(\g_1)\oplus Z^{\cw}_{\alpha}(\g_2)$.

It is enough to verify that $Z^{\cw}_{\alpha}(\g_i)\subseteq
Z^{\cw}_{\alpha}(\g)$. By Corollary 5.5 and an argument similar to
Step $1$, we deduce that\vspace{-.2cm}
\[\dim(HL_2^{\alpha}(\g))=\dim(HL_2^{\alpha}(\f{\g}{Z^{\cw}_{\alpha}(\g_i)}))
-\dim(Z^{\cw}_{\alpha}(\g_i))\vspace{-.1cm},\]from which we have $
Z^{\cw}_{\alpha}(\g_i)\subseteq Z^{\cw}_{\alpha}(\g)$.

{\it Step} $3$. The completion of the proof.

If the one of Hom-Leibniz algebras $(\g_1,\alpha_{\g_1})$ or
$(\g_2,\alpha_{\g_2})$ is abelian, the result obtains from Step
$1$. So suppose that both are non-abelian. Then, using again Step
$1$, there are non-abelian-subalgebras $(\T_i,\alpha_{\T_i})$ of
$(\g_i,\alpha_{\g_i})$, $i=1,2$, such that
$(\g,\alpha_{\g})=(\T_1\oplus\T_2\oplus{\frak
a},\alpha_{\oplus})$,
$Z^{\cw}_{\alpha}(\g)=Z^{\cw}_{\alpha}(\T_1)\oplus
Z^{\cw}_{\alpha}(\T_2)$,
$Z^{\cw}_{\alpha}(\g_i)=Z^{\cw}_{\alpha}(\T_i)$ and
$Z(\T_i)\subseteq[\T_i,\T_i]$. But by Step $2$,
$Z^{\cw}_{\alpha}(\T_1\oplus\T_2)=Z^{\cw}_{\alpha}(\T_1)\oplus
Z^{\cw}_{\alpha}(\T_2)$. We therefore conclude that
$Z^{\cw}_{\alpha}(\g_1\oplus\g_2)=Z^{\cw}_{\alpha}(\g_1)\oplus
Z^{\cw}_{\alpha}(\g_2)$, as desired.
\end{proof}
The following example shows that the regular condition is
essential in Theorem 5.10.
\begin{example}
Let $H(n)$ be the Heisenberg algebra of dimension $2n+1$. Then by
\cite[Theorem 3.5]{EP}, the Hom-Leibniz algebra
$(H(n),id_{H(n)})$ for $n\geq2$ is not capable. But if we define $K=span\{t\}$
and $\alpha_K(t)=0$, then thanks to Theorem \ref{nonperfect}, we infer that
$(H(n),id_{H(n)})\oplus(\kk,\alpha_\kk)$ is capable.
\end{example}

\end{document}